\documentclass[11pt, letterpaper, oneside,reqno]{amsart}

\usepackage{latexsym, amsmath, amssymb, amsfonts, amscd,bm}
\usepackage{mathtools}
\usepackage{amsthm}
\usepackage{t1enc}
\usepackage[mathscr]{eucal}
\usepackage{indentfirst}
\usepackage{graphicx, pb-diagram}
\usepackage{fancyhdr}
\usepackage{fancybox}
\usepackage{enumerate}
\usepackage{xcolor}
\usepackage[all]{xy}
\usepackage{hyperref}
\usepackage{tikz-cd}
\usepackage{blkarray}
\usepackage{url}
\usepackage{float}
\usepackage{bm}
\usepackage[makeroom]{cancel}
\usepackage{accents}

\theoremstyle{plain}
\newtheorem{thm}{Theorem}[section]

\newtheorem{prop}[thm]{Proposition}
\newtheorem*{propnonumber}{Proposition}
\newtheorem*{main}{Main Theorem}

\newtheorem{lemma}[thm]{Lemma}
\newtheorem{cor}[thm]{Corollary}

\newtheoremstyle{underline}
{}        
{}              
{}              
{}    
{\large}              
{:}             
{1mm}         
{{\underline{\thmname{#1}\thmnumber{ #2}}}}  

\theoremstyle{underline}
\newtheorem*{claim*}{Claim}

\theoremstyle{definition}
\newtheorem{defi}[thm]{Definition}

\theoremstyle{remark}
\newtheorem{remark}[thm]{Remark}
\newtheorem{ex}[thm]{Example}

\newtheorem*{ack}{Acknowledgements}

\newcommand*{\dt}[1]{%
	\accentset{\mbox{\large\bfseries .}}{#1}}

\definecolor{forest}{rgb}{0,0.5,0}

\begin{document}
	
	\title[On deformations of coisotropic submanifolds with fixed foliation]{On deformations of coisotropic submanifolds with fixed characteristic foliation}
	\author{Stephane Geudens}
	\address{University College London, Department of Mathematics, 25 Gordon Street, London WC1H 0AY, United Kingdom}
	\email{s.geudens@ucl.ac.uk}
	
	\begin{abstract}
		It is well-known that the deformation problem of a compact coisotropic submanifold $C$ in a symplectic manifold is obstructed in general. We show that it becomes unobstructed if one only allows coisotropic deformations whose characteristic foliation is diffeomorphic to that of $C$. This extends an unobstructedness result in the setting of integral coisotropic submanifolds due to Ruan.
	\end{abstract}
	
	\maketitle
	
	\setcounter{tocdepth}{1} 
	\tableofcontents

	\section*{Introduction}
	
	Lagrangian submanifolds are fundamental objects in symplectic geometry with a simple deformation theory. Given a Lagrangian submanifold $L\subset(M,\omega)$, Weinstein's Lagrangian neighborhood theorem \cite{lagr} states that one can identify $(M,\omega)$ with the cotangent bundle $(T^{*}L,\omega_{can})$ around $L$. This implies that deformations of $L$ correspond with small closed one-forms on $L$. Hence, the Lagrangian deformation problem is linear and unobstructed.
	
	Coisotropic submanifolds encompass the Lagrangian ones, and these submanifolds show up naturally in various contexts (e.g. zero level sets of moment maps, first class constraints in mechanics). A coisotropic submanifold $C\subset(M,\omega)$ carries a characteristic foliation $\mathcal{F}$, which plays a key role in the deformation problem of $C$. Gotay's theorem \cite{gotay} provides an extension of Weinstein's Lagrangian neighborhood theorem, showing that one can identify $(M,\omega)$ with $(U,\Omega_G)$ around $C$. Here $U$ is a neighborhood of the zero section of $T^{*}\mathcal{F}$, and $\Omega_G$ is a symplectic form constructed out of a complement $G$ to $T\mathcal{F}$. In contrast to the Lagrangian case, the coisotropic sections of $(U,\Omega_G)$ are cut out by a highly non-linear equation, which is actually the Maurer-Cartan equation of a suitable $L_{\infty}$-algebra \cite{oh-park},\cite{fiberwise}. Moreover, the coisotropic deformation problem is obstructed in general, meaning that there may exist first order deformations which are not tangent to any path of deformations \cite{marco}.
	
	\vspace{0.2cm}
	
	In this note, we revisit the obstructedness problem for compact coisotropic submanifolds. It appears that obstructedness of the coisotropic deformation problem of $C$ is intimately related with instability of its characteristic foliation. This is suggested by the following two results, going in opposite directions:
	\begin{itemize}
		\item Zambon constructed an obstructed example in \cite{marco}, featuring a compact coisotropic submanifold $C$ with arbitrarily $\mathcal{C}^{1}$-small deformations whose characteristic foliation is not diffeomorphic to that of $C$.
		\item A compact coisotropic submanifold is called integral if its characteristic foliation $\mathcal{F}$ is given by the fibers of a fiber bundle $C\rightarrow C/\mathcal{F}$. Ruan showed that the deformation problem of an integral coisotropic submanifold --within the class of such submanifolds-- is unobstructed \cite{ruan}.
	\end{itemize}
	These results are our motivation to look at a restricted version of the coisotropic deformation problem of $C$, which only allows deformations whose characteristic foliation is diffeomorphic to that of $C$. In other words, we deform $C$ within the class
	\[
	\text{Def}_{\mathcal{F}}(C):=\big\{C'\subset(M,\omega)\ \text{coisotropic}:\ \exists\ \text{foliated diffeomorphism}\ (C,\mathcal{F})\overset{\sim}{\rightarrow}(C',\mathcal{F}')\big\}.
	\]
	Our main result states that this deformation problem is unobstructed (see Thm.~\ref{thm:main}).
	
	\begin{main}
		Let $C\subset(M,\omega)$ be a compact coisotropic submanifold with characteristic foliation $\mathcal{F}$. The deformation problem of $C$ within $\text{Def}_{\mathcal{F}}(C)$ is unobstructed.
	\end{main}
	
	We would like to interpret this unobstructedness result as follows. Zambon's example \cite{marco} shows that the space of compact coisotropic submanifolds of a symplectic manifold $(M,\omega)$ is not smooth in general. More precisely, it may fail to be a Fr\'{e}chet submanifold of the Fr\'{e}chet manifold consisting of all compact submanifolds of $M$. However, if we define an equivalence relation on the space of compact coisotropic submanifolds of $(M,\omega)$, declaring $(C,\mathcal{F})$ and $(C',\mathcal{F}')$ to be equivalent if there exists a foliated diffeomorphism between them, then our Main Theorem states that the resulting equivalence classes are smooth, in a loose way. It would be interesting to know if this can be made sense of in a more precise way.

	\bigskip
	
	\textbf{Overview of the paper.}
	
	Section~\ref{sec:one} collects some background information about coisotropic submanifolds and the Gotay normal form. We recall that, as a consequence of this normal form, first order deformations of a coisotropic submanifold $C$ are leafwise closed one-forms on its characteristic foliation $\mathcal{F}$. With the aim of this note in mind, we also review the well-known fact that first order deformations of the foliation $\mathcal{F}$, modulo those induced by isotopies, are given by the first foliated cohomology group with values in the normal bundle $H^{1}(\mathcal{F};N\mathcal{F})$. 
	
	\vspace{0.2cm}
	
	In Section~\ref{sec:two}, we argue what are the appropriate first order deformations when deforming a coisotropic submanifold $C$ inside $\text{Def}_{\mathcal{F}}(C)$. By the above, these are closed foliated one-forms which in a way should give rise to a trivial class in $H^{1}(\mathcal{F};N\mathcal{F})$. This assignment is achieved by a tool which is central in this paper, namely a canonical transverse differentiation map
	\[
	d_{\nu}:H^{1}(\mathcal{F})\rightarrow H^{1}(\mathcal{F};N^{*}\mathcal{F}).
	\]
	We should note here that since $\mathcal{F}$ is transversely symplectic, there is a canonical isomorphism $H^{1}(\mathcal{F};N\mathcal{F})\cong H^{1}(\mathcal{F};N^{*}\mathcal{F})$. Hence, the map $d_{\nu}$ does indeed give a way to relate first order deformations of $C$ --as a coisotropic submanifold-- with first order deformations of its foliation $\mathcal{F}$. We obtain the following description for the first order deformations of $C$ within the class $\text{Def}_{\mathcal{F}}(C)$ (see Lemma~\ref{inf} and Prop.~\ref{prop:rephrase-canonical}).
	
	\begin{propnonumber}
		Let $C\subset(M,\omega)$ be a compact coisotropic submanifold. When deforming $C$ inside $\text{Def}_{\mathcal{F}}(C)$, a first order deformation is a closed foliated one-form $\beta\in\Omega^{1}(\mathcal{F})$ whose cohomology class lies in the kernel of $d_{\nu}:H^{1}(\mathcal{F})\rightarrow H^{1}(\mathcal{F};N^{*}\mathcal{F}).$
	\end{propnonumber}

	First order deformations have the following geometric interpretation (see Lemma~\ref{lem:extension}): they are exactly the foliated one-forms $\beta\in\Omega^{1}(\mathcal{F})$ admitting an extension $\widetilde{\beta}\in\Omega^{1}(C)$ that satisfies $d\widetilde{\beta}\in\Omega^{2}_{bas}(C)$. Here $\Omega_{bas}(C)$ are the differential forms on $C$ that are basic with respect to $\mathcal{F}$; morally, they are the differential forms on the leaf space $C/\mathcal{F}$. 
	
	In case the coisotropic submanifold $C$ is integral, then our notion of first order deformation reduces to closed foliated one-forms $\beta\in\Omega^{1}(\mathcal{F})$ defining flat sections of the vector bundle $(\mathcal{H}^{1},\nabla)$ made up by the first cohomology groups of the fibers of $C\rightarrow C/\mathcal{F}$ (see Ex.~\ref{ex:integral}). These are exactly the first order deformations considered by Ruan in the deformation problem of integral coisotropic submanifolds \cite{ruan}.
	
	\vspace{0.2cm}

	Section~\ref{sec:three} is the core of this note, as it contains the proof of the Main Theorem. We show that every first order deformation $\beta\in\Omega^{1}(\mathcal{F})$ of $C$ --as defined in the above Proposition-- is tangent to a path of coisotropic sections $\alpha_t$ of the Gotay local model $(U,\Omega_G)$, whose characteristic foliation is diffeomorphic to $\mathcal{F}$. The proof is purely geometric, its main ingredient being a Moser argument (see Lemma~\ref{ruanProp5}).
	In case the coisotropic submanifold $(C,\mathcal{F})$ is integral, our notion of first order deformation recovers that of Ruan (see above), and the path $\alpha_t$ obtained via the Main Theorem again consists of integral coisotropic submanifolds. So our Main Theorem extends Ruan's result \cite{ruan} stating that the deformation problem of an integral coisotropic submanifold --within the class of such submanifolds-- is unobstructed.
	
	\vspace{0.2cm}
	
	At last, Section~\ref{sec:four} discusses two implications that our Main Theorem has for the classical coisotropic deformation problem, which allows \emph{all} coisotropic deformations of $C$.
	
	First, the Main Theorem yields a partial unobstructedness result in this context. It shows that first order deformations of $C$ --i.e. closed foliated one-forms on $\mathcal{F}$-- are unobstructed, if their cohomology class lies in $\ker(d_{\nu})$. We like to interpret this result in the following way.
	\newline
	It is clear geometrically that first order deformations $\beta\in\Omega^{1}(\mathcal{F})$ admitting a closed extension are unobstructed, since the latter yields a symplectic vector field on $(U,\Omega_G)$ whose flow generates a path of coisotropic sections prolonging $\beta$. In practice however, this procedure does not yield many unobstructed first order deformations: it only concerns those closed $\beta\in\Omega^{1}(\mathcal{F})$ whose cohomology class lies in the image of the restriction map $H^{1}(C)\rightarrow H^{1}(\mathcal{F})$, which is a finite dimensional subspace.
	Since the existence of a closed extension implies that $[\beta]$ lies in the kernel of $d_{\nu}$, the partial unobstructedness statement resulting from our Main Theorem extends the geometric fact just mentioned. By contrast, the kernel of $d_{\nu}$ is infinite dimensional in general (see Ex.~\ref{ex:inf}), hence our partial unobstructedness result has the potential to generate much more unobstructed first order deformations.

	Second, we already remarked that there is an $L_{\infty}$-algebra \cite{lada} governing the coisotropic deformation problem of $C$ \cite{oh-park},\cite{fiberwise}. It is given by $\big(\Gamma(\wedge T^{*}\mathcal{F}),\{\lambda_k\}\big)$, and it has the property that its Maurer-Cartan elements correspond with coisotropic sections of the Gotay local model $(U,\Omega_G)$. The multibrackets $\lambda_k$ are constructed out of $\Omega_G$, in particular they depend on the choice of complement $G$ to $T\mathcal{F}$. The $L_{\infty}$-algebra gives a way to detect obstructed first order deformations of $C$. Namely, for a first order deformation to be unobstructed, it needs to define a class in the kernel of the Kuranishi map
	\[
	Kr:H^{1}(\mathcal{F})\rightarrow H^{2}(\mathcal{F}):[\beta]\mapsto[\lambda_2(\beta,\beta)].
	\]
	The partial unobstructedness result following from the Main Theorem implies that $\ker(d_{\nu})$ is contained in $\ker(Kr)$. We double-check that this is indeed the case (see Cor.~\ref{cor:kur}). Actually, we prove a more interesting result. We show that, although the multibrackets $\lambda_k$ depend on the choice of complement $G$, the Kuranishi map has a canonical description in terms of the map $d_{\nu}$ and the induced presymplectic form on $C$. This highlights once more the role played by the transverse differentiation map $d_{\nu}$ in the coisotropic deformation problem, and it implies in particular that $\ker(d_{\nu})$ is contained in $\ker(Kr)$.

	\begin{ack}
		Part of this work was done while visiting the Max Planck Institute for Mathematics. I would like to thank the institute for its hospitality and the financial support. I also acknowledge support from the UCL Institute for Mathematics \& Statistical Science (IMSS). I would like to thank Marco Zambon and Ioan M\u{a}rcu\c{t} for useful conversations. I am particularly grateful to Marco Zambon for sharing his forthcoming work \cite{rel}. 
	\end{ack}
	
	\section{Background and setup}\label{sec:one}
	
	In this section, we first collect some background material about coisotropic submanifolds in symplectic geometry. We then introduce in detail the deformation problem that we will consider in this note. Namely, we want to study small deformations of a given coisotropic submanifold $C\subset(M,\omega)$, whose characteristic foliation is diffeomorphic with that of $C$. At last, we recall the necessary preliminaries about infinitesimal deformations of coisotropic submanifolds and foliations.
	
	\subsection{Coisotropic submanifolds and the Gotay normal form}
	\vspace{0.1cm}
	\noindent
	
	We recall what it means for a submanifold $C$ of a symplectic manifold $(M,\omega)$ to be coisotropic. The definition involves the symplectic orthogonal $TC^{\omega}$, which is defined by
	\[
	TC^{\omega}:=\{v\in TM|_{C}:\omega(v,w)=0\ \ \forall w\in TC\}.
	\]
	We also recall Gotay's theorem \cite{gotay}, which provides a normal form for $(M,\omega)$ around a coisotropic submanifold $C$. Such a normal form is useful to study small deformations of $C$.
	
	\begin{defi}
		A submanifold $C$ of $(M,\omega)$ is called \textbf{coisotropic} if $TC^{\omega}\subset TC$. 
	\end{defi}
	
	We list some alternative characterisations of coisotropic submanifolds $C\subset(M,\omega)$. They involve the pullback $\omega_C\in\Omega^{2}(C)$ of $\omega$ to $C$, and are proved by straightforward linear algebra.
	
	\begin{lemma}\label{equiv}
		For any submanifold $C^{k}\subset(M^{2n},\omega)$, the following are equivalent:
		\begin{enumerate}[i)]
			\item $C^{k}\subset(M^{2n},\omega)$ is coisotropic.
			\item $\omega_{C}$ has constant rank equal to $2(k-n)$.
			\item $\omega_{C}^{k-n+1}=0$.
		\end{enumerate}
	\end{lemma}
	
	The pullback $\omega_{C}$ is a closed two-form of constant rank, so it defines a \textbf{presymplectic structure} on $C$. Consequently, a coisotropic submanifold $C$ has an induced foliation $\mathcal{F}$, defined by $T\mathcal{F}=\ker(\omega_{C})$. We refer to $\mathcal{F}$ as the \textbf{characteristic foliation} of $C\subset(M,\omega)$.
	
	\bigskip
	
	Gotay's theorem \cite{gotay} provides a normal form for a symplectic manifold $(M,\omega)$ around a coisotropic submanifold $C$. The local model lives on a neighborhood of the zero section of the vector bundle $p:T^{*}\mathcal{F}\rightarrow C$, and it is defined as follows. Choose a complement $G$ to $T\mathcal{F}$ inside $TC$, and let $j:T^{*}\mathcal{F}\hookrightarrow T^{*}C$ be the induced inclusion. Denoting by $\omega_{can}$ the canonical symplectic form on $T^{*}C$, the closed two-form
	\[
	\Omega_{G}:=p^{*}\omega_{C}+j^{*}\omega_{can}
	\]
	is non-degenerate along $C$, hence it is a symplectic form on a neighborhood $U$ of $C\subset T^{*}\mathcal{F}$. 
	\begin{thm}[Gotay \cite{gotay}]
		If $C\subset(M,\omega)$ is a coisotropic submanifold, then a neighborhood of $C$ in $M$ is symplectomorphic with the local model $(U,\Omega_G)$. This symplectomorphism can be chosen such that it restricts to the identity map on $C$.			
	\end{thm}
	
	The notation just introduced will be used throughout the paper:
	
	\boxed{
		\parbox{400pt}{
			\begin{itemize}
				\item $G$ is a complement to the characteristic distribution $T\mathcal{F}$ of $C\subset(M,\omega)$.
				\item $j:T^{*}\mathcal{F}\hookrightarrow T^{*}C$ extends an element of $T^{*}\mathcal{F}$ by zero on $G$. We also denote the induced map on differential forms by $j:\Omega^{\bullet}(\mathcal{F})\hookrightarrow\Omega^{\bullet}(C)$.
				\item $(U,\Omega_G)$ is the Gotay local model for the choice of complement $G$.
			\end{itemize}
	}} 
	\vspace{0.2cm}
	
	By Gotay's theorem, studying $\mathcal{C}^{1}$-small deformations of $C\subset(M,\omega)$ amounts to studying small sections of the local model $(U,\Omega_G)$ whose graph is coisotropic. The following result gives a convenient description for the coisotropic sections of $(U,\Omega_G)$. 
	The statement and its proof use the following pieces of notation.
	A section $\alpha\in\Gamma(U)$ defines a diffeomorphism onto its image, which we denote by $\tau_\alpha:C\rightarrow \text{Graph}(\alpha)$. We set $i_{\alpha}:\text{Graph}(\alpha)\hookrightarrow U$ to be the inclusion map, so we have $\alpha=i_{\alpha}\circ\tau_\alpha$ as maps $C\rightarrow U$.
	
	\begin{prop}\label{F}
		Let $C^k\subset(M^{2n},\omega)$ be coisotropic with Gotay local model $(U,\Omega_{G})$.
		For any section $\alpha\in\Gamma(U)$, we have
		\begin{equation}\label{mapF}
			\alpha^{*}\Omega_G=\omega_{C}-d(j(\alpha)).
		\end{equation}
		Consequently, the following are equivalent:
		\begin{enumerate}[i)]
			\item $\text{Graph}(\alpha)\subset(U,\Omega_{G})$ is coisotropic.
			\item $\big(\alpha^{*}\Omega_G\big)^{k-n+1}=0$.
			\item $\omega_{C}-d(j(\alpha))\in\Omega^{2}(C)$ has constant rank, equal to the rank of $\omega_{C}$.
		\end{enumerate}
	\end{prop}
	\begin{proof}
		We first check that the equality \eqref{mapF} holds. Denoting by $\theta_{can}$ the tautological one-form on $T^{*}C$, we have
		\begin{equation*}
			\alpha^{*}\Omega_{G}=\alpha^{*}(p^{*}\omega_{C}-j^{*}d\theta_{can})=\omega_{C}-d(j\circ\alpha)^{*}\theta_{can}=\omega_{C}-d(j(\alpha)).
		\end{equation*}
		Here the second equality uses that $p\circ\alpha=\text{Id}_{C}$, and the third equality holds by the defining property of the tautological one-form (see \cite[Prop.~3.4]{cannas}), which states that the section $j\circ\alpha\in\Gamma(T^{*}C)$ pulls back the tautological one-form $\theta_{can}\in\Omega^{1}(T^{*}C)$ to $j(\alpha)\in\Omega^{1}(C)$.
		
		We now show that $i)$, $ii)$ and $iii)$ are equivalent. By Lemma \ref{equiv},  $\text{Graph}(\alpha)\subset(U,\Omega_G)$ is coisotropic exactly when the pullback $i_{\alpha}^{*}\Omega_G\in\Omega^{2}(\text{Graph}(\alpha))$ has constant rank equal to $2(k-n)$, or equivalently, when  $(i_{\alpha}^{*}\Omega_G)^{k-n+1}=0$. Since the map $\tau_\alpha:C\rightarrow\text{Graph}(\alpha)$ is a diffeomorphism, the former statement is equivalent with $\tau_\alpha^{*}(i_{\alpha}^{*}\Omega_G)=\alpha^{*}\Omega_{G}$ having constant rank $2(k-n)$, and the latter with $(\alpha^{*}\Omega_{G})^{k-n+1}=\tau_\alpha^{*}((i_{\alpha}^{*}\Omega_G)^{k-n+1})=0$. Along with the equality \eqref{mapF}, this proves that $i)$, $ii)$ and $iii)$ are indeed equivalent.
	\end{proof}
	
	Prop.~\ref{F} shows that a coisotropic deformation of $C$ defines an exact deformation of the presymplectic structure $\omega_{C}$, i.e. we have an assignment
	\begin{equation}\label{Fdefs}
		\text{Def}_{U}(C)\rightarrow\text{Def}(\omega_{C}):\alpha\mapsto \omega_{C}-d(j(\alpha)).
	\end{equation}
	Instead of working with coisotropic sections of $(U,\Omega_G)$, it is more convenient to work with the associated presymplectic forms, because these are all defined on the same manifold $C$.

	\subsection{Deformations with fixed characteristic foliation}
	
	\vspace{0.1cm}
	\noindent
	
	Given a coisotropic submanifold $C\subset(M,\omega)$ with characteristic foliation $\mathcal{F}$, our aim is to study deformations of $C$ whose characteristic foliation is diffeomorphic to $\mathcal{F}$. We will now introduce this deformation space, and a local model for it which conveniently describes small deformations.
	
	\begin{defi}\label{def:Defs}
		We define $\text{Def}_{\mathcal{F}}(C)$ to be the space of coisotropic submanifolds $C'$ of $(M,\omega)$ for which there exists a foliated diffeomorphism $(C',\mathcal{F}')\overset{\sim}{\rightarrow} (C,\mathcal{F})$.
	\end{defi}
	
	We are interested in elements of $\text{Def}_{\mathcal{F}}(C)$ $\mathcal{C}^{1}$-close to $C$. That is, passing to the Gotay local model $(U,\Omega_{G})$ of $C$, we look at coisotropic sections $\alpha\in\Gamma(U)$ for which $(\text{Graph}(\alpha),\ker(i_{\alpha}^{*}\Omega_{G}))$ and $(C,\ker\omega_{C})$ are diffeomorphic as foliated manifolds. In this respect, note the following.
	
	\begin{lemma}\label{lem:compose}
		For any coisotropic section $\alpha$ of $(U,\Omega_G)$, the following are equivalent:
		\begin{enumerate}[i)]
			\item There exists a foliated diffeomorphism $\psi:(C,\ker\omega_{C})\rightarrow(\text{Graph}(\alpha),\ker(i_{\alpha}^{*}\Omega_{G}))$.
			\item There exists a foliated diffeomorphism $\phi:(C,\ker\omega_{C})\rightarrow\big(C,\ker(\omega_{C}-d(j(\alpha)))\big)$.
		\end{enumerate}
	\end{lemma}
	\begin{proof}
		Recall from Prop. \ref{F} that the section $\alpha$ defines a diffeomorphism $\tau_\alpha:C\rightarrow\text{Graph}(\alpha)$ with inverse $p|_{\text{Graph}(\alpha)}:\text{Graph}(\alpha)\rightarrow C$, satisfying $\tau_\alpha^{*}(i_{\alpha}^{*}\Omega_{G})=\omega_{C}-d(j(\alpha))$. Consequently, given $\psi$, we define $\phi:=p|_{\text{Graph}(\alpha)}\circ\psi$. Conversely, given $\phi$, we set $\psi:=\tau_\alpha\circ\phi$. 
	\end{proof}
	
	The lemma above motivates the following definition, which introduces the local model $\text{Def}^{U}_{\mathcal{F}}(C)$ for the space $\text{Def}_{\mathcal{F}}(C)$. This is the key object of study in this note.

	\begin{defi}\label{definition}
		Let $C\subset (M,\omega)$ be a coisotropic submanifold with characteristic foliation $\mathcal{F}$ and Gotay local model $(U,\Omega_G)$. We set $\text{Def}^{U}_{\mathcal{F}}(C)$ to be the space of $\alpha\in\Gamma(U)$ such that
		\begin{equation*}
			\begin{cases}
				\text{Graph}(\alpha)\ \text{is coisotropic in}\ (U,\Omega_{G})\\
				\text{There is a foliated diffeomorphism}\  (C,\ker\omega_{C})\overset{\sim}{\rightarrow}\big(C,\ker(\omega_{C}-d(j(\alpha)))\big)
			\end{cases}.
		\end{equation*}
	\end{defi}
	
	\begin{remark}
		The existence of a foliated diffeomorphism $(C,\ker\omega_{C})\overset{\sim}{\rightarrow}\big(C,\ker(\omega_{C}-d(j(\alpha)))\big)$ actually implies that $\text{Graph}(\alpha)$ is coisotropic in $(U,\Omega_G)$, because of Prop.~\ref{F}. Nevertheless, we prefer to include the first requirement in Def.~\ref{definition} for the sake of clarity.
	\end{remark}
	
	In conclusion, the map \eqref{Fdefs}, combined with the map which takes a constant rank two-form to its kernel, yields a 2-step process
	\begin{equation*}
		\text{Def}_{U}(C)\rightarrow\text{Def}(\omega_{C})\overset{\ker}{\rightarrow}\text{Def}(\mathcal{F}):\alpha\mapsto\ker(\omega_{C}-d(j(\alpha))),
	\end{equation*}
	which assigns to a coisotropic deformation $\alpha\in\Gamma(U)$ of $C$ a deformation of its characteristic foliation $\mathcal{F}$. We restrict attention to those elements $\alpha\in\text{Def}_{U}(C)$ for which the image under this map is equivalent with $\mathcal{F}$ by means of a diffeomorphism.
	
	\begin{ex}\label{ex:contact}
		We present a class of coisotropic submanifolds that have an abundance of deformations with diffeomorphic characteristic foliation. If $C\subset(M^{2n},\omega)$ is coisotropic of codimension $q$, then $C$ is called \textbf{of q-contact type} \cite{bolle} if there exist $\alpha_1,\ldots,\alpha_q\in\Omega^{1}(C)$ such that:
		\begin{enumerate}
			\item $d\alpha_i=\omega_{C}$ for $i=1,\ldots,q$.
			\item $\alpha_1\wedge\ldots\wedge\alpha_q\wedge\omega_{C}^{n-q}$ is nowhere zero on $C$.
		\end{enumerate}
		We give some concrete examples. First, a Lagrangian torus $C\subset(M^{2n},\omega)$ is of $n$-contact type with $\alpha_i=d\theta_i$, where $(\theta_1,\ldots,\theta_n)$ are the angular coordinates on the torus $C$. Second, the unit sphere $S^{3}\subset(\mathbb{R}^{4},\omega_{can})$ is of $1$-contact type with $\alpha$ equal to the usual connection one-form on $S^{3}$ for the Hopf fibration, i.e.
		\[
		\alpha=i^{*}\left(\sum_{j=1}^{2}x_jdy_j-y_jdx_j\right).
		\]
		There is a normal form around coisotropic submanifolds of $q$-contact type \cite[Lemma 1]{bolle} which refines Gotay's theorem, stating that a neighborhood of $C$ in $(M,\omega)$ is symplectomorphic with the model
		\begin{equation}\label{eq:contactmodel}
			\Big(C\times B_{\epsilon}^{q},p^{*}\omega_{C}+\sum_{i=1}^{q}d(y_i p^{*}\alpha_i)\Big).
		\end{equation}
		Here $B_{\epsilon}^{q}$ is the $\epsilon$-ball in $\mathbb{R}^{q}$ centered at the origin, $(y_1,\ldots,y_q)$ are the coordinates on $B_{\epsilon}^{q}$ and $p:C\times B_{\epsilon}^{q}\rightarrow C$ is the projection. As a consequence of this normal form, all slices $C\times\{h\}$ for small enough $h\in B_{\epsilon}^{q}$ are coisotropic deformations of $C$ whose characteristic foliation is diffeomorphic with that of $C$. Indeed, pulling back the symplectic form \eqref{eq:contactmodel} to $C\times\{h\}$, we obtain
		\[
		i_{h}^{*}\Big(p^{*}\omega_{C}+\sum_{i=1}^{q}d(y_i p^{*}\alpha_i)\Big)=\Big(1+\sum_{i=1}^{q}h_i\Big)p|_{C\times\{h\}}^{*}\omega_C.
		\]
		This shows that, as long as $\|h\|$ is small enough so that $1+\sum_{i=1}^{q}h_i$ is nonzero, then $C\times\{h\}$ is coisotropic, and the projection $p|_{C\times\{h\}}:C\times\{h\}\rightarrow C$ is a diffeomorphism which matches the characteristic foliations of $C\times\{h\}$ and $C$.
	\end{ex}

	\subsection{First order deformations of coisotropic submanifolds and foliations}
	
	\vspace{0.1cm}
	\noindent
	
	There are two classes of geometric objects whose deformation theory is important in this note, namely coisotropic submanifolds and foliations. At the infinitesimal level, the deformations of these objects are governed by certain cochain complexes and their associated cohomology groups, which we now recall.

	\subsubsection{Coisotropic submanifolds}\label{subsub:coiso}
	Let $C\subset(M,\omega)$ be a coisotropic submanifold with characteristic foliation $\mathcal{F}$. It is well-known that the complex governing infinitesimal deformations of $C$ is the leafwise de Rham complex $\big(\Omega^{\bullet}(\mathcal{F}),d_{\mathcal{F}}\big)$, see \cite{oh-park},\cite{equivalences}. Here $\Omega^{k}(\mathcal{F}):=\Gamma(\wedge^{k}T^{*}\mathcal{F})$, and the differential $d_{\mathcal{F}}$ is defined by the usual Koszul formula
	\begin{align*}
		d_{\mathcal{F}}\alpha(V_0,\ldots,V_k)&=\sum_{i=0}^{k}(-1)^{i}V_i\big(\alpha(V_0,\ldots,V_{i-1},\widehat{V_i},V_{i+1},\ldots,V_k)\big)\\
		&\hspace{0.5cm}+\sum_{i<j}(-1)^{i+j}\alpha\big([V_i,V_j],V_0,\ldots,\widehat{V_i},\ldots,\widehat{V_j},\ldots,V_k\big).
	\end{align*}
	We spell this out in the following lemma, which was already obtained in \cite[Cor.~2.5]{equivalences} using Poisson geometry. We give a simple alternative argument, which only relies on Prop.~\ref{F}.
	
	\begin{lemma}\label{lem:infcoiso}
		Let $C^{k}\subset(M^{2n},\omega)$ be coisotropic with Gotay local model $(U,\Omega_G)$. If $\alpha_t$ is a smooth one-parameter family of coisotropic sections of $U$ starting at the zero section $\alpha_0=0$, then $\dt{\alpha_0}$ is closed with respect to $d_{\mathcal{F}}$.
	\end{lemma}
	\begin{proof}
		By Prop. \ref{F}, we have
		\begin{align}\label{eq:differentiate}
			0&=\left.\frac{d}{dt}\right|_{t=0}(\alpha_{t}^{*}\Omega_G)^{k-n+1}\nonumber\\
			&=(k-n+1)\left(\left.\frac{d}{dt}\right|_{t=0}\alpha_{t}^{*}\Omega_G\right)\wedge(\alpha_{0}^{*}\Omega_G)^{k-n}\nonumber\\
			&=(k-n+1)\left(\left.\frac{d}{dt}\right|_{t=0}\big(\omega_{C}-d(j(\alpha_t))\big)\right)\wedge\omega_{C}^{k-n}\nonumber\\
			&=-(k-n+1)d(j(\dt{\alpha_0}))\wedge\omega_{C}^{k-n},
		\end{align}
		also using that the pullback of $\Omega_G$ to the zero section is $\omega_C$. Since $C$ is coisotropic, we know that $k\geq n$, hence \eqref{eq:differentiate} implies that $d(j(\dt{\alpha_0}))\wedge\omega_{C}^{k-n}=0$. So for all $V,W\in\Gamma(T\mathcal{F})$, we have
		\[
		0=\iota_{W}\iota_V\big(d(j(\dt{\alpha_0}))\wedge\omega_{C}^{k-n}\big)=(d_{\mathcal{F}}\dt{\alpha_0})(V,W)\omega_{C}^{k-n}.
		\]
		By Lemma \ref{equiv} $ii)$, the rank of $\omega_{C}$ is $2(k-n)$, so that $\omega_{C}^{k-n}$ is nowhere zero. So we necessarily have that $(d_{\mathcal{F}}\dt{\alpha_0})(V,W)$ vanishes, which shows that $d_{\mathcal{F}}\dt{\alpha_0}=0$.
	\end{proof}
	
	This result motivates the following definition.
	
	\begin{defi}
		Assume that $C\subset(M,\omega)$ is a coisotropic submanifold with characteristic foliation $\mathcal{F}$ and Gotay local model $(U,\Omega_G)$. 
		\begin{enumerate}[i)]
			\item A \textbf{first order deformation} of $C$ is a foliated one-form $\beta\in\Omega^{1}(\mathcal{F})$ which is $d_{\mathcal{F}}$-closed. 
			\item A first order deformation $\beta$ is said to be \textbf{unobstructed} if there exists a smooth path $\alpha_t$ of coisotropic sections of $U$ starting at the zero section $\alpha_0=0$, satisfying $\dt{\alpha_0}=\beta$. Otherwise, we say that $\beta$ is \textbf{obstructed}.
		\end{enumerate}
	\end{defi}
	
	It is well-known that a coisotropic submanifold $C$ has obstructed first order deformations in general, reflecting the fact that the space of coisotropic submanifolds may fail to be smooth around $C$. See \cite{marco} for an example where this occurs.
	
	
	\subsubsection{Foliations}\label{subsub:fol} Let $\mathcal{F}$ be a foliation on a manifold $C$. The normal bundle $N\mathcal{F}:=TC/T\mathcal{F}$ carries a flat $T\mathcal{F}$-connection, called the Bott connection, which is defined by
	\[
	\nabla_{X}\overline{Y}=\overline{[X,Y]},\hspace{1cm} X\in\Gamma(T\mathcal{F}), \overline{Y}\in\Gamma(N\mathcal{F}).
	\]
	We obtain a complex $\big(\Omega^{\bullet}(\mathcal{F};N\mathcal{F}),d_{\nabla}\big)$, where $\Omega^{k}(\mathcal{F};N\mathcal{F}):=\Gamma(\wedge^{k}T^{*}\mathcal{F}\otimes N\mathcal{F})$ are foliated forms with values in $N\mathcal{F}$, and the differential $d_{\nabla}$ is defined by
	\begin{align*}
		d_{\nabla}\eta(V_0,\ldots,V_k)&=\sum_{i=0}^{k}(-1)^{i}\nabla_{V_i}\big(\eta(V_0,\ldots,V_{i-1},\widehat{V_i},V_{i+1},\ldots,V_k)\big)\\
		&\hspace{0.5cm}+\sum_{i<j}(-1)^{i+j}\eta\big([V_i,V_j],V_0,\ldots,\widehat{V_i},\ldots,\widehat{V_j},\ldots,V_k\big).
	\end{align*}
	
	The work \cite{heitsch} of Heitsch shows that infinitesimal deformations of the foliation $\mathcal{F}$ are one-cocycles in  $\big(\Omega^{\bullet}(\mathcal{F};N\mathcal{F}),d_{\nabla}\big)$. Moreover, if a smooth deformation of $\mathcal{F}$ is obtained applying an isotopy to $\mathcal{F}$, then the corresponding infinitesimal deformation is a one-coboundary in $\big(\Omega^{\bullet}(\mathcal{F};N\mathcal{F}),d_{\nabla}\big)$. We now spell this out in a bit more detail.
	
	Assume we are given a smooth path of foliations $\mathcal{F}_{t}$ with $\mathcal{F}_{0}=\mathcal{F}$. We fix a complement $G$ to $T\mathcal{F}$ and identify $G\cong N\mathcal{F}$. The induced Bott connection on $G$ is given by
	\[
	\nabla_{X}Y=\text{pr}_{G}[X,Y],\hspace{1cm} X\in\Gamma(T\mathcal{F}), Y\in\Gamma(G),
	\]
	where $\text{pr}_{G}:TC\rightarrow G$ is the projection. Say that $C$ is compact, then there exists $\epsilon>0$ such that $T\mathcal{F}_t$ is still transverse to $G$ for $0\leq t\leq\epsilon$. We can therefore assume that 
	\[
	T\mathcal{F}_t=\text{Graph}(\eta_t)=\{X+\eta_t(X): X\in\Gamma(T\mathcal{F})\}
	\]
	for some $\eta_t\in\Gamma(T^{*}\mathcal{F}\otimes G)$. The next result is essentially \cite[Cor.~2.11]{heitsch} and \cite[Prop.~2.12]{heitsch}. 
	
	\begin{lemma}\label{lem:inffol}
		In the setup described above, we have:
		\begin{enumerate}
			\item The infinitesimal deformation $\dt{\eta_0}$ is closed with respect to $d_{\nabla}$.
			\item If the path $\mathcal{F}_t$ is generated by an isotopy $(\phi_t)$, i.e. $T\mathcal{F}_t=(\phi_t)_{*}T\mathcal{F}$, then the corresponding infinitesimal deformation is exact. Indeed,
			\[
			\dt{\eta_0}=d_{\nabla}(\text{pr}_{G}V_0),
			\]
			where $(V_t)$ is the time-dependent vector field corresponding with the isotopy $(\phi_t)$.
		\end{enumerate}
	\end{lemma}
	
	\begin{remark}
		Actually, \cite[Prop.~2.12]{heitsch} only concerns infinitesimal deformations arising from a path of foliations generated by the flow of a (time-independent) vector field. However, it is clear that the proof still works when the vector field is time-dependent, and the resulting statement is part $(2)$ of Lemma \ref{lem:inffol} above.
	\end{remark}
	
	Lemma \ref{lem:inffol} justifies the following definition.
	
	\begin{defi}\label{def:inffol}
		Let $\mathcal{F}$ be a foliation on a manifold $C$.
		\begin{enumerate}[i)]
			\item A \textbf{first order deformation} of $\mathcal{F}$ is an element $\eta\in\Omega^{1}(\mathcal{F};N\mathcal{F})$ which is $d_{\nabla}$-closed.
			\item We call $\mathcal{F}$ \textbf{infinitesimally rigid} if the cohomology group $H^{1}(\mathcal{F};N\mathcal{F})$ vanishes.
		\end{enumerate}
	\end{defi}
	
	Let us also mention here that the Bott connection $\nabla$ on $N\mathcal{F}$ induces a flat $T\mathcal{F}$-connection $\nabla^{*}$ on $N^{*}\mathcal{F}\cong T\mathcal{F}^{0}$. The two are related by the Leibniz rule
	\begin{equation}\label{eq:Leibnizconnections}
		X\langle \overline{Y},\beta\rangle=\langle\nabla_{X}\overline{Y},\beta\rangle+\langle \overline{Y},\nabla^{*}_{X}\beta\rangle,\hspace{0.5cm}X\in\Gamma(T\mathcal{F}),\overline{Y}\in\Gamma(N\mathcal{F}),\beta\in\Gamma(N^{*}\mathcal{F}).
	\end{equation}
	Of particular interest to us is the case in which the foliation $\mathcal{F}$ is transversely symplectic, i.e. $T\mathcal{F}=\ker\omega_C$ for a closed two-form $\omega_C\in\Omega^{2}(C)$. Then we have a vector bundle isomorphism $\omega_C^{\flat}:N\mathcal{F}\rightarrow N^{*}\mathcal{F}$, and it was noticed in \cite[Lemma 5.2]{koszul} that this isomorphism is compatible with the flat connections $\nabla$ and $\nabla^{*}$, i.e.
	\[
	\nabla^{*}_{X}\big(\omega_C^{\flat}(\overline{Y})\big)=\omega_C^{\flat}\big(\nabla_{X}\overline{Y}\big),\hspace{0.5cm} X\in\Gamma(T\mathcal{F}),\overline{Y}\in\Gamma(N\mathcal{F}).
	\]
	It follows that $\omega_C^{\flat}$ induces an isomorphism of complexes
	\[
	\text{Id}\otimes\omega_C^{\flat}:\big(\Gamma(\wedge^{\bullet}T^{*}\mathcal{F}\otimes N\mathcal{F}),d_{\nabla}\big)\overset{\sim}{\longrightarrow}\big(\Gamma(\wedge^{\bullet}T^{*}\mathcal{F}\otimes N^{*}\mathcal{F}),d_{\nabla^{*}}\big), 
	\]
	hence also an isomorphism in cohomology $ H^{\bullet}(\mathcal{F};N\mathcal{F})\overset{\sim}{\rightarrow} H^{\bullet}(\mathcal{F};N^{*}\mathcal{F})$.

	\section{First order deformations and the transverse differentiation map}\label{sec:two}
	
	In this section, we investigate what happens at the infinitesimal level when deforming a coisotropic submanifold $C$ inside the class $\text{Def}_{\mathcal{F}}(C)$. We argue that first order deformations of $C$ are leafwise closed one-forms $\alpha\in\Omega^{1}(\mathcal{F})$ whose cohomology class $[\alpha]$ lies in the kernel of a certain transverse differentiation map $d_{\nu}:H^{1}(\mathcal{F})\rightarrow H^{1}(\mathcal{F};N^{*}\mathcal{F})$.
	
	\subsection{Smooth paths in $\mathbf{\text{Def}^{U}_{\mathcal{F}}(C)}$}
	\vspace{0.1cm}
	\noindent
	
	In order to linearize the conditions in Def. \ref{definition}, we first need to specify what are smooth paths in the local deformation space $\mathbf{\text{Def}^{U}_{\mathcal{F}}(C)}$. We will start by defining a notion of smooth path in $\text{Def}_{\mathcal{F}}(C)$, which we use to induce a notion of smooth path in $\mathbf{\text{Def}^{U}_{\mathcal{F}}(C)}$.
	
	\begin{defi}\label{def:smoothness}
		A path $C_t$ in $\text{Def}_{\mathcal{F}}(C)$ is smooth if there is a smooth path of embeddings $\Phi_t:C\hookrightarrow M$ such that $C_{t}=\Phi_t(C)$ and $\Phi_t:(C,\mathcal{F})\rightarrow (C_t,\mathcal{F}_{t})$ is a foliated diffeomorphism. 
	\end{defi}
	
	In what follows, it is crucial that a smooth path in $\text{Def}_{\mathcal{F}}(C)$ comes with a smooth family of foliated diffeomorphisms. We comment some more on Def.~\ref{def:smoothness} in the following remark.

	\begin{remark}\label{smoothness}
		Note that there is a $1:1$ correspondence between the spaces
		\[
		\text{Emb}_{\mathcal{F}}(C):=\left\{\text{Embeddings}\ \Phi:C\hookrightarrow M:\ \begin{cases}
			\Phi(C)\ \text{is coisotropic}\\
			\iota_{v}\Phi^{*}\omega=0\ \ \forall v\in T\mathcal{F}
		\end{cases}\right\}
		\]
		and
		\[
		\text{Pairs}_{\mathcal{F}}(C):=\left\{(C',\phi)\in \text{Def}_{\mathcal{F}}(C)\times\text{Diff}(C,C')|\ \phi:(C,\mathcal{F})\overset{\sim}{\rightarrow}(C',\mathcal{F}')\right\}.
		\]
		There is an obvious notion of smooth path in $\text{Emb}_{\mathcal{F}}(C)$, while $\text{Pairs}_{\mathcal{F}}(C)$ clearly surjects onto $\text{Def}_{\mathcal{F}}(C)$. We obtain our notion of smooth path in $\text{Def}_{\mathcal{F}}(C)$, by declaring a path $(C_t)$ in $\text{Def}_{\mathcal{F}}(C)$ to be smooth if it lifts to a smooth path in $\text{Pairs}_{\mathcal{F}}(C)\cong\text{Emb}_{\mathcal{F}}(C)$.
	\end{remark}
	
	We now make a choice of tubular neighborhood of $C$ via Gotay's theorem, and we restrict to $\mathcal{C}^{1}$-small deformations of $C$ that stay inside this neighborhood $(U,\Omega_{G})$.
	
	\begin{defi}\label{def:smoothpath}
		A path $\alpha_t$ in $\text{Def}^{U}_{\mathcal{F}}(C)$ is smooth if there is a smooth path of embeddings $\Phi_t:C\hookrightarrow U$ such that $\text{Graph}(\alpha_t)=\Phi_t(C)$ and $p\circ\Phi_t:(C,\ker\omega_{C})\rightarrow\big(C,\ker(\omega_{C}-d(j(\alpha_t)))\big)$ is a foliated diffeomorphism.
	\end{defi}
	
	\begin{remark}\label{rem:pairs}
		As in Rem.~\ref{smoothness}, there is a $1:1$ correspondence between the spaces
		\[
		\text{Emb}_{\mathcal{F}}^{U}:=\left\{\text{Embeddings}\ \Phi:C\hookrightarrow U:\ \begin{cases}
			\Phi(C)=\text{Graph}(\alpha)\ \text{for some}\ \alpha\in\Gamma(U)\\
			\text{Graph}(\alpha)\ \text{is coisotropic}\\
			\iota_{v}\Phi^{*}\Omega_G=0\ \ \forall v\in T\mathcal{F}
		\end{cases}\right\}
		\]
		and
		\[
		\text{Pairs}_{\mathcal{F}}^{U}(C):=\left\{(\alpha,\phi)\in \text{Def}_{\mathcal{F}}^{U}(C)\times\text{Diff}(C)|\ \phi:(C,\ker\omega_{C})\overset{\sim}{\rightarrow}\big(C,\ker(\omega_{C}-d(j(\alpha)))\big)\right\}.
		\]
		Explicitly, the correspondence is given by
		\[
		\text{Emb}_{\mathcal{F}}^{U}\rightarrow \text{Pairs}_{\mathcal{F}}^{U}(C):
		\Phi\rightarrow\big(\Phi\circ(p\circ\Phi)^{-1},p\circ\Phi\big),
		\]
		with inverse
		\[
		\text{Pairs}_{\mathcal{F}}^{U}(C)\rightarrow\text{Emb}_{\mathcal{F}}^{U}:(\alpha,\phi)\rightarrow\alpha\circ\phi.
		\]
		Again, there is an obvious notion of smooth path in $\text{Emb}_{\mathcal{F}}^{U}$, while $\text{Pairs}_{\mathcal{F}}^{U}(C)$ clearly surjects onto $\text{Def}_{\mathcal{F}}^{U}(C)$. We then obtain our notion of smooth path in $\text{Def}_{\mathcal{F}}^{U}(C)$, by declaring a path $(\alpha_t)$ in $\text{Def}_{\mathcal{F}}^{U}(C)$ to be smooth if it lifts to a smooth path in $\text{Pairs}_{\mathcal{F}}^{U}(C)\cong\text{Emb}_{\mathcal{F}}^{U}(C)$.
	\end{remark}

	\subsection{First order deformations}
	\vspace{0.1cm}
	\noindent
	
	Given a smooth path $\alpha_t$ in $\text{Def}^{U}_{\mathcal{F}}(C)$ deforming $C$, we will now figure out what properties are satisfied by the corresponding infinitesimal deformation $\dt{\alpha_0}$. We first give a non-canonical description of infinitesimal deformations in terms of a certain chain map. We then rephrase this description in a canonical way by passing to cohomology.
	
	\subsubsection{A provisional definition} One would expect that in some way, the infinitesimal deformation $\dt{\alpha_0}$ gives rise to a trivial infinitesimal deformation of the foliation $\mathcal{F}$, namely a $1$-coboundary in the Bott complex $\big(\Omega^{\bullet}(\mathcal{F};G),d_{\nabla}\big)$.
	We claim that this happens by means of the following map, which we denote provisionally by $\Phi$.
	
	\begin{defi}\label{phi}
		Let $\Phi:\Gamma(\wedge^{k}T^{*}\mathcal{F})\rightarrow\Gamma(\wedge^{k}T^{*}\mathcal{F}\otimes G)$ denote the map defined by
		\[
		\left\langle \Phi(\alpha)(V_{1},\ldots,V_{k}),\beta\right\rangle=d(j(\alpha))\big(V_{1},\ldots,V_{k},(\omega_{C}^{\flat})^{-1}(\beta)\big)
		\]
		for $\alpha\in\Gamma(\wedge^{k}T^{*}\mathcal{F}), \beta\in\Gamma(G^{*})$ and $V_{1},\ldots,V_{k}\in\Gamma(T\mathcal{F})$.
	\end{defi}
	
	In other words, the map $\Phi$ is defined as follows. For $\alpha\in\Gamma(\wedge^{k}T^{*}\mathcal{F})$, we have that
	\[
	d(j(\alpha))\in\Gamma(\wedge^{k+1}T^{*}\mathcal{F})\oplus\Gamma(\wedge^{k}T^{*}\mathcal{F}\otimes G^{*})\oplus\Gamma(\wedge^{k-1}T^{*}\mathcal{F}\otimes\wedge^{2}G^{*}).
	\] 
	The map $\Phi$ picks the component in $\Gamma(\wedge^{k}T^{*}\mathcal{F}\otimes G^{*})$, and then applies the isomorphism 
	\[
	-\text{Id}\otimes(\omega_{C}^{\flat})^{-1}:\wedge^{k}T^{*}\mathcal{F}\otimes G^{*}\rightarrow \wedge^{k}T^{*}\mathcal{F}\otimes G.
	\]
	
	\begin{lemma}\label{inf}
		Let $C\subset(M,\omega)$ be a compact coisotropic submanifold with local model $(U,\Omega_G)$. Assume that $\alpha_t$ is a smooth curve in $\text{Def}_{\mathcal{F}}^{U}(C)$ passing through $C$ at time $t=0$. Then the infinitesimal deformation $\dt{\alpha_0}$ satisfies the following:
		\begin{enumerate}[i)]
			\item $d_{\mathcal{F}}\dt{\alpha_0}=0$, 
			\item $\Phi(\dt{\alpha_0})$ is exact in  $\big(\Omega^{\bullet}(\mathcal{F};G),d_{\nabla}\big)$.
		\end{enumerate}
	\end{lemma} 
	\begin{proof}
		We already know that item $i)$ holds by Lemma \ref{lem:infcoiso}, so we only have to prove $ii)$. Compactness of $C$ implies that there exists $\epsilon>0$ such that $\ker\big(\omega_{C}-d(j(\alpha_t))\big)\subset T\mathcal{F}\oplus G$ is still transverse to $G$ for all $0\leq t\leq\epsilon$.  This means that there exist $\eta_t\in\Gamma(T^{*}\mathcal{F}\otimes G)$ such that $\ker\big(\omega_{C}-d(j(\alpha_t))\big)=\text{Graph}(\eta_{t})$ for all $t\in[0,\epsilon]$. 
		So for all $v\in T\mathcal{F}$, we have that $v+\eta_t(v)\in\ker\big(\omega_{C}-d(j(\alpha_t))\big)$, it follows that
		\[
		\big(\omega_{C}-d(j(\alpha_t))\big)^{\flat}(\eta_t(v))=d(j(\alpha_t))^{\flat}(v).
		\]
		Differentiating at $t=0$, we get 
		\[
		\omega_{C}^{\flat}\big(\dt{\eta_0}(v)\big)=d(j(\dt{\alpha_0}))^{\flat}(v).
		\]
		Consequently, for any $\beta\in G^{*}$, we have
		\begin{align*}
			\big\langle\Phi(\dt{\alpha_0})(v),\beta\big\rangle&=d(j(\dt{\alpha_0}))\big(v,(\omega_{C}^{\flat})^{-1}(\beta)\big)\\
			&=\big\langle \omega_{C}^{\flat}\big(\dt{\eta_0}(v)\big),(\omega_{C}^{\flat})^{-1}(\beta)\big\rangle\\
			&=\big\langle -\dt{\eta_0}(v),\beta\big\rangle,
		\end{align*}
		showing that $\Phi(\dt{\alpha_0})=-\dt{\eta_0}$. It remains to argue that $\dt{\eta_0}$ is a coboundary in $\big(\Omega^{\bullet}(\mathcal{F};G),d_{\nabla}\big)$.
		Since $\alpha_t$ is a smooth path in $\text{Def}^{U}_{\mathcal{F}}(C)$, Def. \ref{def:smoothpath} guarantees that there exists a smooth family $(\phi_t)\in\text{Diff}(C)$ such that
		\[
		\phi_t:(C,T\mathcal{F})\overset{\sim}{\rightarrow}\big(C,\text{Graph}(\eta_{t})\big).
		\]
		Precomposing $\phi_t$ with $\phi_{0}^{-1}$, we can assume that $\phi_0=\text{Id}_{C}$, so that the family of foliations $\text{Graph}(\eta_t)$ for $t\in[0,\epsilon]$ is generated by applying an isotopy to $T\mathcal{F}$. Part $(2)$ of Lemma \ref{lem:inffol} now implies that $\dt{\eta_0}$ is indeed exact in $\big(\Omega^{\bullet}(\mathcal{F};G),d_{\nabla}\big)$. This finishes the proof. 
	\end{proof}
	
	The above lemma motivates the following provisional definition.
	
	\begin{defi}[Provisional]\label{def:provisional}
		When deforming $C$ inside $\text{Def}_{\mathcal{F}}(C)$, a first order deformation is a foliated one-form $\beta\in\Omega^{1}(\mathcal{F})$ such that $d_{\mathcal{F}}\beta=0$ and $\Phi(\beta)$ is exact in $\big(\Omega^{\bullet}(\mathcal{F};G),d_{\nabla}\big)$.
	\end{defi}

\subsubsection{A canonical definition}\label{subsubsec:canonical}
	Def.~\ref{def:provisional} is not entirely satisfactory, since it makes reference to the chosen complement $G$. In what follows, we derive an equivalent definition which is completely canonical. This is done by showing that $\Phi$ is a chain map (up to sign), and that the induced map in cohomology is canonical (i.e. independent of the complement $G$). The proofs of these statements use ingredients from the spectral sequence of the foliation $\mathcal{F}$. We defer the proofs to \S\ref{sec:appendix_transdiff} in the Appendix, to avoid disrupting the flow of the paper. 
	
	\bigskip
	
	The description of the map in cohomology induced by $\Phi$ involves the following well-known operation in foliation theory. Recall that
	any foliation $\mathcal{F}$ comes with a \textbf{transverse differentiation map}, denoted by
	\begin{equation}\label{eq:dnu}
		d_{\nu}:H^{\bullet}(\mathcal{F})\rightarrow H^{\bullet}(\mathcal{F};N^{*}\mathcal{F}).
	\end{equation}
	This map is constructed as follows \cite{osorno}. If $r:\Omega^{\bullet}(C)\rightarrow\Omega^{\bullet}(\mathcal{F})$ denotes the restriction to the leaves of $\mathcal{F}$, then there is a short exact sequence of complexes
	\[
	0\rightarrow \left(\Omega^{\bullet}_{\mathcal{F}}(C),d\right)\hookrightarrow\left(\Omega^{\bullet}(C),d\right)\overset{r}{\rightarrow}\left(\Omega^{\bullet}(\mathcal{F}),d_{\mathcal{F}}\right)\rightarrow 0.
	\]
	It induces a long exact sequence in cohomology
	\[
	\cdots\rightarrow H^{k}_{\mathcal{F}}(C)\rightarrow H^{k}(C)\rightarrow H^{k}(\mathcal{F})\overset{\mathfrak{d}}{\rightarrow} H^{k+1}_{\mathcal{F}}(C)\rightarrow H^{k+1}(C)\rightarrow H^{k+1}(\mathcal{F})\rightarrow\cdots.
	\]
	The connecting homomorphism $\mathfrak{d}$ is defined as
	\begin{equation}\label{connecting}
		\mathfrak{d}:H^{k}(\mathcal{F})\rightarrow H^{k+1}_{\mathcal{F}}(C):[\alpha]\mapsto[d\widetilde{\alpha}],
	\end{equation}
	where $\widetilde{\alpha}\in\Omega^{k}(C)$ is any extension of $\alpha\in\Omega^{k}(\mathcal{F})$. Next, let
	$p:\Omega^{k+1}_{\mathcal{F}}(C)\rightarrow\Omega^{k}(\mathcal{F},N^{*}\mathcal{F})$ denote the map characterised by
	\[
	\left\langle p(\beta)(V_1,\ldots,V_k),\overline{N}\right\rangle=\beta(V_1,\ldots,V_k,N)
	\]
	for $V_1,\ldots,V_k\in\Gamma(T\mathcal{F})$ and $\overline{N}\in\Gamma(N\mathcal{F})$. This map is well-defined because forms in $\Omega^{\bullet}_{\mathcal{F}}(C)$ vanish when evaluated on elements of $\Gamma(T\mathcal{F})$. Since $p$ commutes with the differentials $d$ and $d_{\nabla^{*}}$, there is an induced map in cohomology 
	\begin{equation}\label{tau}
		[p]:H^{k+1}_{\mathcal{F}}(C)\rightarrow H^{k}(\mathcal{F};N^{*}\mathcal{F}).
	\end{equation}
	The transverse differentiation map \eqref{eq:dnu} is obtained by composing the maps \eqref{connecting} and \eqref{tau}.
	
	\begin{defi}
	The \textbf{transverse differentiation map} $d_{\nu}$ is defined by
	\[
	d_{\nu}:H^{k}(\mathcal{F})\rightarrow H^{k}(\mathcal{F};N^{*}\mathcal{F}):[\alpha]\mapsto[p(d\tilde{\alpha})],
	\]
	where $\widetilde{\alpha}\in\Omega^{k}(C)$ is any extension of $\alpha\in\Omega^{k}(\mathcal{F})$.
	\end{defi}  

\begin{remark}
	The map $d_{\nu}$ appears for instance in the study of symplectic foliations $(\mathcal{F},\omega)$. There $d_{\nu}[\omega]\in H^{2}(\mathcal{F};N^{*}\mathcal{F})$ measures the transverse variation of the leafwise symplectic form $\omega\in\Omega^{2}(\mathcal{F})$, which plays a role in the integrability problem for the associated regular Poisson structure \cite[\S 5]{crainic-fernandes}.
\end{remark}
	
	The following result allows us to describe first order deformations of $C$ inside $\text{Def}_{\mathcal{F}}(C)$ in a canonical way. For the proof, we refer to \S\ref{sec:appendix_transdiff} in the Appendix.
	
	\begin{prop}\label{prop:rephrase-canonical}
	\begin{enumerate}[i)]
	\item The map
		$$
		\Phi:\big(\Omega^{\bullet}(\mathcal{F}),d_{\mathcal{F}}\big)\rightarrow\big(\Omega^{\bullet}(\mathcal{F};G),d_{\nabla}\big)
		$$
		is a chain map, up to sign. That is, $\Phi\circ d_{\mathcal{F}}=-d_{\nabla}\circ\Phi$.		
  \item The map $[\Phi]$ induced in cohomology is canonical, since it agrees with
  \[
  \big[-\text{Id}\otimes(\omega_{C}^{\flat})^{-1}\big]\circ d_{\nu}:H^{\bullet}(\mathcal{F})\rightarrow H^{\bullet}(\mathcal{F};N\mathcal{F}).
  \]
  \end{enumerate}
\end{prop}

Because the map $\big[-\text{Id}\otimes(\omega_{C}^{\flat})^{-1}\big]:H^{\bullet}(\mathcal{F};N^{*}\mathcal{F})\rightarrow H^{\bullet}(\mathcal{F};N\mathcal{F})$ appearing in Prop.~\ref{prop:rephrase-canonical} is an isomorphism, we can now rephrase Def.~\ref{def:provisional} as follows.

\begin{defi}\label{def:firstorderdefs}
	When deforming $C$ inside $\text{Def}_{\mathcal{F}}(C)$, a first order deformation is a foliated one-form $\beta\in\Omega^{1}(\mathcal{F})$ such that $d_{\mathcal{F}}\beta=0$ and the cohomology class $[\beta]\in H^{1}(\mathcal{F})$ lies in the kernel of the transverse differentiation map $d_{\nu}:H^{1}(\mathcal{F})\rightarrow H^{1}(\mathcal{F};N^{*}\mathcal{F})$.
\end{defi}
	
We finish this section by giving a geometric description of first order deformations. For the proof, see again \S\ref{sec:appendix_transdiff} in the Appendix.

\begin{lemma}\label{lem:extension}
	Let $C$ be a manifold with a foliation $\mathcal{F}$. For any $\beta\in\Omega^{1}(\mathcal{F})$, the following are equivalent:
	\begin{enumerate}
		\item $d_{\mathcal{F}}\beta=0$ and $d_{\nu}[\beta]=0$,
		\item There exists an extension $\widetilde{\beta}\in\Omega^{1}(C)$ of $\beta$ such that $\iota_{V}d\widetilde{\beta}=0$ for all $V\in\Gamma(T\mathcal{F})$.
	\end{enumerate}
\end{lemma}

Recall that a differential form $\eta$ on a foliated manifold $(C,\mathcal{F})$ is called \textbf{basic} if $\iota_{V}\eta=0$ and $\pounds_{V}\eta=0$ for all $V\in\Gamma(T\mathcal{F})$. If the foliation $\mathcal{F}$ is simple, this means that $\eta$ descends to the leaf space $C/\mathcal{F}$. The basic differential forms constitute a subcomplex $(\Omega^{\bullet}_{bas}(C),d)$ of the de Rham complex.	
Statement $(2)$ in Lemma~\ref{lem:extension} says that  $d\tilde{\beta}\in\Omega^{2}_{bas}(C)$. 	
	
\begin{ex}[The integral case]\label{ex:integral}
	Let $C$ be a compact coisotropic submanifold whose characteristic foliation $\mathcal{F}$ is given by the fibers of a smooth fiber bundle $pr:C\rightarrow C/\mathcal{F}$. Such coisotropic submanifolds are called \textbf{integral}, and they are studied in \cite{ruan}. To facilitate the computation of first order deformations of such $C$, it is useful to note that the cohomology groups $H^{1}(\mathcal{F})$ and $H^{1}(\mathcal{F};N^{*}\mathcal{F})$ have the following convenient descriptions \cite[Thm.~I.5.2]{hoster}.
	
	\begin{itemize}
		\item The first cohomology groups of the fibers of $pr$ constitute a vector bundle $\mathcal{H}^{1}$ over $C/\mathcal{F}$, i.e. 
		\begin{equation}\label{eq:vb}
			\mathcal{H}^{1}_{q}=H^{1}(pr^{-1}(q))\hspace{0.5cm}\forall q\in C/\mathcal{F},
		\end{equation}
		and $H^{1}(\mathcal{F})$ can be identified canonically with the space of sections $\Gamma(\mathcal{H}^{1})$. Namely, the class $[\beta]\in H^{1}(\mathcal{F})$ corresponds with the section $\tau_{\beta}\in\Gamma(\mathcal{H}^{1})$ given by
		\[
		\tau_{\beta}:q\mapsto\left[\beta|_{pr^{-1}(q)}\right].
		\]
		Moreover, the vector bundle $\mathcal{H}^{1}$ carries a natural flat connection $\nabla$, called the Gauss-Manin connection, which is defined as follows. Denoting by $\mathfrak{X}(C)^{\mathcal{F}}$ the Lie subalgebra of $(\mathfrak{X}(C),[\cdot,\cdot])$ consisting of projectable vector fields,
		\[
		\mathfrak{X}(C)^{\mathcal{F}}:=\left\{Y\in\mathfrak{X}(C):[Y,\Gamma(T\mathcal{F})]\subset\Gamma(T\mathcal{F})\right\},
		\]
		we have a short exact sequence
		\begin{equation}\label{seq}
			0\rightarrow\Gamma(T\mathcal{F})\hookrightarrow\mathfrak{X}(C)^{\mathcal{F}}\rightarrow\mathfrak{X}(C/\mathcal{F})\rightarrow 0.
		\end{equation}
		The Gauss-Manin connection is then defined as follows, for $V\in\mathfrak{X}(C/\mathcal{F})$:
		\[
		\nabla_{V}\tau_{\beta}:=\tau_{\pounds_{\widetilde{V}}\beta},
		\]
		where $\widetilde{V}\in\mathfrak{X}(C)^{\mathcal{F}}$ is any lift of $V$ in \eqref{seq}. The connection is well-defined because of Cartan's magic formula, and it is flat because $[\widetilde{V},\widetilde{W}]$ is a lift of $[V,W]$ whenever $\widetilde{V},\widetilde{W}\in\mathfrak{X}(C)^{\mathcal{F}}$ are lifts of $V,W\in\mathfrak{X}(C/\mathcal{F})$. 
		\item The cohomology group $H^{1}(\mathcal{F},N^{*}\mathcal{F})$ can be identified with the space of one-forms on $C/\mathcal{F}$ with values in the vector bundle $\mathcal{H}^{1}$:
		\begin{equation}\label{eq:isomconormal}
			H^{1}(\mathcal{F};N^{*}\mathcal{F})\cong\Gamma(T^{*}(C/\mathcal{F})\otimes\mathcal{H}^{1}).
		\end{equation}
	\end{itemize}
	
	We now compute the kernel of $d_{\nu}:H^{1}(\mathcal{F})\rightarrow H^{1}(\mathcal{F},N^{*}\mathcal{F})$. Fix a class $[\beta]\in H^{1}(\mathcal{F})$, and let $\widetilde{\beta}\in\Omega^{1}(C)$ be any extension of $\beta$. Denoting by $r:\Omega^{1}(C)\rightarrow\Omega^{1}(\mathcal{F})$ the restriction map and using the isomorphism \eqref{eq:isomconormal}, we have
	\begin{align*}
		[\beta]\in\ker(d_{\nu})&\Leftrightarrow r(\iota_{Y}d\widetilde{\beta})\ \text{is foliated exact for all}\ Y\in\mathfrak{X}(C)^{\mathcal{F}},\\
		&\Leftrightarrow\pounds_{Y}\beta\ \text{is foliated exact for all}\ Y\in\mathfrak{X}(C)^{\mathcal{F}},\\
		&\Leftrightarrow \tau_{\beta}\in\Gamma(\mathcal{H}^{1})\ \text{is flat w.r.t.}\ \nabla. 
	\end{align*}
	Hence, according to Def.~\ref{def:firstorderdefs}, first order deformations of an integral coisotropic submanifold are closed foliated one-forms that define flat sections of the vector bundle $\big(\mathcal{H}^1,\nabla\big)$.
	
	In \cite{ruan}, Ruan studies the deformation problem of an integral coisotropic submanifold $C$, within the class of integral coisotropic submanifolds. He shows that first order deformations in this case are indeed closed foliated one-forms on $C$ which define flat sections of $\big(\mathcal{H}^1,\nabla\big)$, see \cite[Lemma 2]{ruan}. Hence, our Def.~\ref{def:firstorderdefs} is consistent with Ruan's work.
\end{ex}

	\section{Unobstructedness of first order deformations}\label{sec:three}
	
	Let $C\subset(M,\omega)$ be a compact coisotropic submanifold. It is known that the coisotropic deformation problem of $C$ is obstructed in general, i.e. there may exist first order deformations of $C$ that are not tangent to any path of deformations \cite{marco}. However, Ruan proved that the deformation problem becomes unobstructed when restricting to integral coisotropic submanifolds \cite{ruan}, i.e. those for which the characteristic foliation is given by a fibration.
	
	This section contains our main result, which states that the deformation problem of a compact coisotropic submanifold $C$ \emph{inside the class $\text{Def}_{\mathcal{F}}(C)$} (see Def.~\ref{def:Defs}) is unobstructed. This is an extension of Ruan's unobstructedness result, since the first order deformations from Def.~\ref{def:firstorderdefs} reduce to those considered by Ruan in case $(C,\mathcal{F})$ is integral (see Ex.~\ref{ex:integral}).
	
	\begin{defi}
		Given a compact coisotropic submanifold $C\subset(M,\omega)$, a first order deformation $\beta$ of $C$ (see Def.~\ref{def:firstorderdefs}) is \textbf{unobstructed} if there exists a smooth path $\alpha_{t}$ in $\text{Def}_{\mathcal{F}}^{U}(C)$ (see Def.~\ref{def:smoothpath}) such that $\alpha_0=0$ and $\dt{\alpha_0}=\beta$. The deformation problem of $C$ inside $\text{Def}_{\mathcal{F}}(C)$ is unobstructed if all first order deformations are unobstructed.
	\end{defi}

	As a first step, we consider the distinguished class of first order deformations consisting of  foliated one-forms $\beta\in\Omega^{1}(\mathcal{F})$ that admit a closed extension $\widetilde{\beta}\in\Omega^{1}(C)$. Such $\beta$ indeed satisfy the requirements of Def.~\ref{def:firstorderdefs}. These first order deformations are easily proved to be unobstructed. The following lemma is just an enhancement of \cite[Rem.~4.6]{equivalences}. 
	
	\begin{lemma}\label{lem:closedext}
		Let $C\subset(M,\omega)$ be compact coisotropic submanifold with characteristic foliation $\mathcal{F}$. If $\beta\in\Omega^{1}(\mathcal{F})$ admits a closed extension $\widetilde{\beta}\in\Omega^{1}(C)$, then $\beta$ is an unobstructed first order deformation of $C$.
	\end{lemma}
	\begin{proof}
		Let $(U,\Omega_{G})$ be the Gotay local model for some choice of splitting $TC=T\mathcal{F}\oplus G$, and denote by $p:U\rightarrow C$ the projection. By assumption, we have a closed one-form $p^{*}\tilde{\beta}$ on $U$, which gives rise to a symplectic vector field
		\[
		X:=(\Omega_{G}^{\flat})^{-1}(-p^{*}\widetilde{\beta}).
		\] 
		Let $(\phi_t)$ denote the flow of $X$, and note that compactness of $C$ implies that there exists $\epsilon>0$ such that flow lines of $X$ starting at points of $C$ exist up to time $\epsilon$. Shrinking $\epsilon$ if necessary, we can assume that the submanifolds $\phi_t(C)$ for $0\leq t<\epsilon$ are graphs of sections $\alpha_t\in\Gamma(U)$. We now check that $\alpha_t$ is a smooth path in $\text{Def}_{\mathcal{F}}^{U}(C)$, as defined in Def.~\ref{def:smoothpath}.
		
		It is clear that $\text{Graph}(\alpha_t)$ is coisotropic, since $C$ is coisotropic and the $\phi_t$ are symplectomorphisms. It remains to check that 
		\begin{equation}\label{eq:toshoww}
			\psi_t:=p\circ\phi_t\circ\alpha_0:(C,\ker\omega_{C})\rightarrow\big(C,\ker(\omega_{C}-d(j(\alpha_t)))\big)
		\end{equation}
		is a foliated diffeomorphism. To do so, note that 
		\[
		\alpha_t=\phi_t\circ\alpha_0\circ\psi_t^{-1}:C\rightarrow U.
		\]
		Since $\phi_t^{*}\Omega_G=\Omega_G$, we have
		\[
		\omega_C=\alpha_0^{*}\Omega_G=\alpha_0^{*}(\phi_t^{*}\Omega_G)=(\phi_t\circ\alpha_0)^{*}\Omega_G,
		\]
		and therefore
		\[
		\psi_{t}^{*}(\alpha_{t}^{*}\Omega_G)=(\alpha_t\circ\psi_t)^{*}\Omega_G=(\phi_t\circ\alpha_0)^{*}\Omega_G=\omega_C.
		\]
		Also invoking Prop.~\ref{F}, this confirms that \eqref{eq:toshoww} is a foliated diffeomorphism. 
		
		Finally, we check that $\beta$ is tangent to the path $\alpha_{t}$.  By \cite[Lemma 3.13]{equivalences}, we have that
		$\dt{\alpha}_{0}$ is the section of $T^{*}\mathcal{F}$ defined by the vertical fiberwise constant vector field $P(X)$ on $U$.
		Here $P$ denotes the restriction to $C$ composed with the vertical projection in the splitting $TU|_{C}=TC\oplus T^{*}\mathcal{F}$. Using Lemma~\ref{dF} in the Appendix, we conclude that
		\begin{equation}\label{eq:compdot}
			\dt{\alpha}_{0}=-r\left(i^{*}\Omega_G^{\flat}(P(X))\right)=-r\left(i^{*}\Omega_G^{\flat}(X)\right)=r(i^{*}(p^{*}\widetilde{\beta}))=\beta.
		\end{equation}
		Here the second equality uses the fact that $\Omega_G(TC,T\mathcal{F})=0$. This finishes the proof.
	\end{proof}
	
	One can use Lemma~\ref{lem:closedext} to show that the deformation problem of $C$ inside $\text{Def}_{\mathcal{F}}(C)$ is unobstructed whenever $C\rightarrow C/\mathcal{F}$ is a fiber bundle admitting a global section. That is to say, a special case of Ruan's unobstructedness result for integral coisotropic submanifolds follows from Lemma~\ref{lem:closedext}. We provide the details in the following example.
	
	\begin{ex}
		Let $C\subset(M,\omega)$ be a compact coisotropic submanifold for which $pr:C\rightarrow C/\mathcal{F}$ is a fiber bundle admitting a global section $\sigma$. If $\beta\in\Omega^{1}(\mathcal{F})$ is a first order deformation of $C$ inside $\text{Def}_{\mathcal{F}}(C)$, then according to Lemma~\ref{lem:extension}, $\beta$ has an extension $\widetilde{\beta}\in\Omega^{1}(C)$ satisfying 
		\[
		d\widetilde{\beta}=pr^{*}\gamma
		\]
		for some $\gamma\in\Omega^{2}(C/\mathcal{F})$. Moreover, since $pr\circ\sigma=\text{Id}$, we have
		\[
		\gamma=\sigma^{*}(pr^{*}\gamma)=\sigma^{*}(d\widetilde{\beta})=d(\sigma^{*}\widetilde{\beta}).
		\]
		This implies that $\widetilde{\beta}-pr^{*}(\sigma^{*}\widetilde{\beta})$ is a closed extension of $\beta$. By Lemma~\ref{lem:closedext}, $\beta$ is unobstructed.
		
		A concrete example of this type is the coisotropic submanifold considered in \cite{marco}. There $C$ is the torus $\mathbb{T}^{4}$ with presymplectic form $\omega_{C}=d\theta_1\wedge d\theta_2$. The characteristic foliation $\mathcal{F}$ is given by the fibers of the trivial $\mathbb{T}^{2}$-bundle
		\[
		\mathbb{T}^{4}\rightarrow\mathbb{T}^{2}:(\theta_1,\theta_2,\theta_3,\theta_4)\mapsto(\theta_1,\theta_2).
		\]
		Choosing $G:=\text{Span}\{\partial_{\theta_1},\partial_{\theta_2}\}$ as a complement, $C$ is coisotropic in the Gotay local model
		\[
		\big(\mathbb{T}^{4}\times\mathbb{R}^{2},d\theta_1\wedge d\theta_2+d\theta_3\wedge d\xi_1+d\theta_4\wedge d\xi_2\big),
		\]
		where $(\xi_1,\xi_2)$ are the fiber coordinates corresponding with the frame $\{d\theta_3,d\theta_4\}$ of $T^{*}\mathcal{F}$.
		It was shown in \cite{marco} that the deformation problem of $C$ --as a coisotropic submanifold-- is obstructed. That is, there exist closed foliated one-forms $\beta\in\Omega^{1}(\mathcal{F})$ that are not tangent to a path of coisotropic deformations. By contrast, the restricted deformation problem of $C$ inside $\text{Def}_{\mathcal{F}}(C)$ is unobstructed. That is, any closed foliated one-form $\beta\in\Omega^{1}(\mathcal{F})$ that gives rise to a class in the kernel of $d_{\nu}:H^{1}(\mathcal{F})\rightarrow H^{1}(\mathcal{F};N^{*}\mathcal{F})$ is tangent to a path of coisotropic deformations, whose characteristic foliation is moreover diffeomorphic with $\mathcal{F}$.
	\end{ex}
	
	We now proceed to the proof of our main theorem. The argument relies on the following lemma, which is based on a result in the forthcoming work \cite{rel}.

	\begin{lemma}\label{ruanProp5}
		Let $C\subset(M,\omega)$ be a compact coisotropic submanifold with Gotay local model $(U,\Omega_G)$. Assume that $\{\eta_t\}_{0\leq t\leq\epsilon}$ is a smooth family of presymplectic structures on $C$ with
		\[
		\begin{cases}
			\eta_0=\omega_C\\
			rk(\eta_t)\ \text{is constant}\ \text{and}\ \ker\eta_t=\ker\omega_C\\
			[\eta_t]\in H^{2}(C)\ \text{is constant}
		\end{cases}.
		\]
		Shrinking $\epsilon$ if necessary, there exist a smooth path $\{\sigma_t\}_{0\leq t\leq\epsilon}\subset\text{Def}_{\mathcal{F}}^{U}(C)$  with $\sigma_0=0$, and a smooth family of diffeomorphisms $\{\psi_t\}_{0\leq t\leq\epsilon}\subset\text{Diff}_{0}(C)$ such that $\psi_{0}=\text{Id}$ and
		\begin{equation}\label{eq:checkpullback}
			\sigma_{t}^{*}\Omega_G=\psi_{t}^{*}\eta_t\hspace{1cm} \text{for all}\ \ 0\leq t\leq\epsilon.
		\end{equation}
	\end{lemma}
	
	The lemma gives sufficient conditions for presymplectic structures close to $\omega_C$ to be obtained from coisotropic sections of the local model $(U,\Omega_G)$, up to isotopy. In other words, it gives presymplectic structures that lie in the image of the map \eqref{Fdefs}, when we quotient the codomain by the equivalence relation given by isotopies.
	
	\begin{proof}[Proof of Lemma \ref{ruanProp5}]
		For each value of $t\in[0,\epsilon]$, we can use $G$ as a complement to $\ker\eta_t$. Constructing the associated Gotay local models, we get a family of symplectic structures
		\[
		\omega_t=p^{*}\eta_t+j^{*}\omega_{can},
		\]
		defined on a neighborhood $U_t$ of $C\subset T^{*}\mathcal{F}$. The tube lemma implies that the intersection $\cap_{t\in [0,\epsilon]}U_t$ is an open neighborhood of $C$; hence shrinking $U$ if necessary, we can assume that all $\omega_t$ are symplectic on $U$ for $t\in[0,\epsilon]$.
		
		Since the inclusion $i:C\hookrightarrow U$ induces an isomorphism in cohomology and the class $[\eta_t]\in H^{2}(C)$ is constant, the same holds for the class $[\omega_t]\in H^{2}(U)$:
		\[
		0=\frac{d}{dt}[\eta_t]=\frac{d}{dt}[i^{*}\omega_t]=[i^{*}]\left[\frac{d}{dt}\omega_t\right]\Rightarrow \left[\frac{d}{dt}\omega_t\right]=0.
		\]
		Choosing primitives $\frac{d}{dt}\omega_t=d\beta_t$, we now apply Moser's argument. We make the ansatz
		\begin{equation}\label{moser}
			0=\frac{d}{dt}\phi_{t}^{*}\omega_t=\phi_{t}^{*}\left(\pounds_{X_t}\omega_t+\frac{d}{dt}\omega_t\right)=\phi_{t}^{*}\left(d\iota_{X_t}\omega_t+d\beta_t\right),
		\end{equation}
		where $(\phi_t)$ is an isotopy with corresponding time-dependent vector field $(X_t)$. We can solve the equation $\iota_{X_t}\omega_t=-\beta_t$ for $X_t\in\mathfrak{X}(U)$. Since $C$ is compact, we can shrink $\epsilon>0$ such that $\phi_{t}^{-1}(C)$ is the graph of a section $\sigma_t\in\Gamma(U)$ for each $0\leq t\leq\epsilon$. For such $t$, the map 
		\[
		f_t=p\circ\phi_{t}^{-1}\circ\sigma_0:C\rightarrow C
		\]
		is a diffeomorphism and we have 
		\[
		\sigma_t=\phi_{t}^{-1}\circ\sigma_0\circ f_{t}^{-1}.
		\]
		
		We claim that the pair $(\sigma_t,\psi_t:=f_t^{-1})$ meets the requirements of the lemma. To show that $\sigma_t$ is a smooth path in $\text{Def}_{\mathcal{F}}^{U}(C)$ (see Def.~\ref{def:smoothpath}), we first check that $\text{Graph}(\sigma_t)\subset(U,\Omega_G)$ is coisotropic. By the first equality in \eqref{moser}, we have
		\[
		\phi_{t}^{*}\omega_t=\omega_0=\Omega_G,
		\]
		and this implies that
		\begin{equation}\label{eq:pullback}
			\sigma_{t}^{*}\Omega_G=(f_{t}^{-1})^{*}(\sigma_{0}^{*}((\phi_{t}^{-1})^{*}\Omega_G))=(f_{t}^{-1})^{*}(\sigma_{0}^{*}\omega_t)=(f_{t}^{-1})^{*}\eta_t.
		\end{equation}
		This shows that the rank of $\sigma_{t}^{*}\Omega_G$ is equal to the rank of $\eta_t$, which by assumption equals the rank of $\omega_{C}$. Hence $\text{Graph}(\sigma_t)\subset(U,\Omega_G)$ is coisotropic, by Prop.~\ref{F}. We also showed at the same time that the equality \eqref{eq:checkpullback} is satisfied. It only remains to check that the map
		\[
		f_t=p\circ\phi_{t}^{-1}\circ\sigma_0:(C,\ker\omega_{C})\rightarrow\big(C,\ker(\omega_{C}-d(j(\sigma_t)))\big)
		\]
		is a foliated diffeomorphism. With Prop.~\ref{F} in mind, the equality \eqref{eq:pullback} shows that
		\[
		f_t^{*}\big(\omega_{C}-d(j(\sigma_t))\big)=\eta_t,
		\]
		hence $f_t$ takes the foliation $\ker\eta_t=\ker\omega_C$ to the foliation $\ker(\omega_{C}-d(j(\sigma_t)))$.
	\end{proof}
	
	At last, we can prove our main result.
	
	\begin{thm}\label{thm:main}
		If $C\subset(M,\omega)$ is a compact coisotropic submanifold, then the deformation problem of $C$ inside $\text{Def}_{\mathcal{F}}(C)$ is unobstructed.
	\end{thm}
	\begin{proof}
		Let $\beta$ be a first order deformation of $C$, as defined in Def.~\ref{def:firstorderdefs}. We will construct a smooth path $\alpha_t$ in $\text{Def}_{\mathcal{F}}^{U}(C)$ such that $\alpha_0=0$ and $\dt{\alpha_0}=\beta$.
		
		By Lemma~\ref{lem:extension}, $\beta$ has an extension $\tilde{\beta}\in\Omega^{1}(C)$ satisfying $\iota_{V}d\tilde{\beta}=0$ for all $V\in\Gamma(T\mathcal{F})$. Consider the family of two-forms $\{\omega_C-td\tilde{\beta}\}_{t\geq 0}$ on $C$. By compactness, there exists $\epsilon>0$ such that $rk(\omega_C-td\tilde{\beta})_{p}\geq rk(\omega_C)_{p}$ for all $0\leq t\leq\epsilon$ and $p\in C$. But we also know that $\ker\omega_C\subset\ker(\omega_C-td\tilde{\beta})$, so $rk(\omega_C-td\tilde{\beta})_{p}\leq rk(\omega_C)_{p}$ for all $p\in C$ and $t\geq 0$. In conclusion, for all $0\leq t\leq\epsilon$, the two-form $\omega_C-td\tilde{\beta}$ is presymplectic with the same kernel as $\omega_C$.
		
		We now apply Lemma \ref{ruanProp5}. Shrinking $\epsilon$ if necessary, there exist a smooth path $\{\sigma_t\}_{0\leq t\leq\epsilon}$ in $\text{Def}_{\mathcal{F}}^{U}(C)$ with $\sigma_0=0$ and a smooth family of diffeomorphisms $\{\psi_t\}_{0\leq t\leq\epsilon}\subset\text{Diff}_{0}(C)$ such that $\psi_0=\text{Id}$ and
		\begin{equation}\label{pullback}
			\psi_{t}^{*}\left(\omega_C-td\tilde{\beta}\right)=\omega_C-d(j(\sigma_t).
		\end{equation}
		Let $(X_t)$ be the time-dependent vector field of $(\psi_t)$. Differentiating the equality \eqref{pullback} gives
		\[
		\psi_{t}^{*}\left(\pounds_{X_t}(\omega_C-td\tilde{\beta})-d\tilde{\beta}\right)=-d\left(j\left(\dt{\sigma_t}\right)\right).
		\]
		Evaluating at $t=0$, we get
		\[
		\pounds_{X_0}\omega_C-d\tilde{\beta}=-d\left(j\left(\dt{\sigma_0}\right)\right)
		\]
		and therefore
		\[
		d\left(\iota_{X_0}\omega_C-\tilde{\beta}+j\left(\dt{\sigma_0}\right)\right)=0.
		\]
		Let us denote the closed one-form between brackets by $\theta\in\Omega^{1}(C)$. It gives rise to a symplectic vector field $Z:=(\Omega_{G}^{\flat})^{-1}(p^{*}\theta)$ on $(U,\Omega_G)$, with flow $(\phi_t)$. Shrinking $\epsilon$ if necessary, we can make sure that the submanifolds $\phi_t(\text{Graph}(\sigma_t))$ are graphs of sections $\alpha_t\in\Gamma(U)$.
		
		We claim that $\alpha_t$ is a smooth path in $\text{Def}_{\mathcal{F}}^{U}(C)$. It is clear that $\text{Graph}(\alpha_t)\subset(U,\Omega_G)$ is coisotropic, since $\text{Graph}(\sigma_t)$ is coisotropic and $\phi_t$ is a symplectomorphism. Because $\sigma_t$ is a smooth path in $\text{Def}_{\mathcal{F}}^{U}(C)$, there exists a smooth path of embeddings $\Phi_t:C\hookrightarrow U$ such that $\text{Graph}(\sigma_t)=\Phi_t(C)$ and 
		\begin{equation}\label{eq:foliateddiffeo}
			p\circ\Phi_t:(C,\ker\omega_{C})\rightarrow\big(C,\ker(\omega_{C}-d(j(\sigma_t)))\big)
		\end{equation}
		is a foliated diffeomorphism. To conclude that $\alpha_t$ is a smooth path in $\text{Def}_{\mathcal{F}}^{U}(C)$, it suffices to check that the embeddings $\phi_t\circ\Phi_t:C\hookrightarrow U$ are such that
		\begin{equation}\label{eq:showdiffeo}
			p\circ\phi_t\circ\Phi_t:(C,\ker\omega_{C})\rightarrow\big(C,\ker(\omega_{C}-d(j(\alpha_t)))\big)
		\end{equation}
		is a foliated diffeomorphism. To do so, we argue as in the proof of Lemma \ref{lem:closedext}. Let us first define diffeomorphisms
		\begin{equation}\label{eq:f_t}
			f_t:=p\circ\phi_t\circ\sigma_t,
		\end{equation}
		so that 
		\begin{equation}\label{eq:alpha_t}
			\alpha_t=\phi_t\circ\sigma_t\circ f_{t}^{-1}.
		\end{equation}
		Because $\phi_{t}^{*}\Omega_G=\Omega_G$, we have
		\[
		\sigma_{t}^{*}\Omega_G=\sigma_{t}^{*}(\phi_{t}^{*}\Omega_G)=(\phi_t\circ\sigma_t)^{*}\Omega_G=(\alpha_t\circ f_t)^{*}\Omega_G=f_{t}^{*}(\alpha_t^{*}\Omega_G),
		\]
		hence by Prop.~\ref{F} this shows that
		\[
		f_t:\big(C,\ker(\omega_{C}-d(j(\sigma_t)))\big)\rightarrow\big(C,\ker(\omega_{C}-d(j(\alpha_t)))\big)
		\]
		is a foliated diffeomorphism. Composing with \eqref{eq:foliateddiffeo}, we get a foliated diffeomorphism
		\[
		f_t\circ p\circ\Phi_t:(C,\ker\omega_C)\rightarrow\big(C,\ker(\omega_{C}-d(j(\alpha_t)))\big).
		\]
		Since by \eqref{eq:f_t}, we have
		\[
		f_t\circ p\circ\Phi_t=(p\circ\phi_t\circ\sigma_t)\circ p\circ\Phi_t=p\circ\phi_t\circ\Phi_t,
		\]
		this confirms that \eqref{eq:showdiffeo} is a foliated diffeomorphism. Hence, $\alpha_t$ is a smooth path in $\text{Def}_{\mathcal{F}}^{U}(C)$.
		
		At last, we check that $\beta$ is tangent to the path $\alpha_t$. We denote by $P$ the vertical projection induced by the splitting $T(T^{*}\mathcal{F})|_{C}=TC\oplus T^{*}\mathcal{F}$. Using the expression \eqref{eq:alpha_t} and the chain rule, we get for $q\in C$ that
		\begin{align}\label{velocity}
			\dt{\alpha_0}(q)&=P\left(\dt{\alpha_0}(q)\right)\nonumber\\
			&=P\left(\left.\frac{d}{dt}\right|_{t=0}(\phi_t\circ\sigma_t\circ f_{t}^{-1})(q)\right)\nonumber\\
			&=P\left(\left.\frac{d}{dt}\right|_{t=0}\phi_t(q)+\left.\frac{d}{dt}\right|_{t=0}\sigma_t(q)+\left.\frac{d}{dt}\right|_{t=0}f_{t}^{-1}(q)\right)\nonumber\\
			&=P\left(Z(q)\right)+\dt{\sigma_0}(q).
		\end{align}
		Here we used that the last summand in the third line above is tangent to $C$. By Lemma \ref{dF} in the Appendix, we have that $P(Z)\in\Gamma(T^{*}\mathcal{F})$ is the restriction to $T\mathcal{F}$ of the one-form $-\theta\in\Omega^{1}(C)$, see the computation \eqref{eq:compdot}. That is,
		\[
		P(Z)=\beta-\dt{\sigma_0},
		\]
		and inserting this equality into \eqref{velocity}, we obtain that $\dt{\alpha_0}(q)=\beta(q)$, as desired.
	\end{proof}

\begin{remark}
The coisotropic sections $\sigma_t$ obtained in Lemma~\ref{ruanProp5} are not uniquely determined since they depend on a choice of primitive, as the proof of Lemma \ref{ruanProp5} shows. Nevertheless, the proof of Thm.~\ref{thm:main} shows that any path $\sigma_t$ thus obtained can be used to prolong a first order deformation $\beta$, up to modifying the $\sigma_t$ suitably by a symplectic isotopy.

Alternatively, instead of using the conclusion of Lemma~\ref{ruanProp5}, one can prove Thm.~\ref{thm:main} by running the proof of Lemma~\ref{ruanProp5} explicitly using a well chosen primitive. Then the coisotropic sections $\sigma_t$ obtained this way do not need to be corrected anymore by a symplectic isotopy. This approach yields a simple algorithm which produces out of a first order deformation $\beta$ a path $\{\sigma_t\}$ in $\text{Def}_{\mathcal{F}}^{U}(C)$ with $\sigma_0=0$ and $\dt{\sigma_0}=\beta$. It consists of the following steps:
\begin{enumerate}
\item Fix an extension $\widetilde{\beta}\in\Omega^{1}(C)$ of $\beta$ as in Lemma~\ref{lem:extension}.
\item For small enough $t$, the two-forms
\[
\omega_t:=p^{*}(\omega_C-td\tilde{\beta})+j^{*}\omega_{can}=\Omega_G-tp^{*}d\tilde{\beta}
\]
are symplectic on a neighborhood of $C\subset T^{*}\mathcal{F}$. There, one can solve the equation
\begin{equation}\label{eq:tosolve}
\iota_{X_t}\omega_t=p^{*}\tilde{\beta}.
\end{equation}
\item Let $(\phi_t)$ be the flow of $X_t$. For small enough $t$, we have $\phi_t^{-1}(C)=\text{Graph}(\sigma_t)$ for a smooth path $\sigma_t$ in $\text{Def}_{\mathcal{F}}^{U}(C)$. Now $\beta=\dt{\sigma_0}$.
\end{enumerate}
We have to justify that $\beta=\dt{\sigma_0}$. Since the time-dependent vector field of the isotopy $(\phi_t^{-1})$ is $Z_t:=-(\phi_t^{-1})_{*}X_t$, it follows from \cite[Lemma~3.13]{equivalences} that $\dt{\sigma_0}$ is the foliated one-form defined by $P(Z_0)=-P(X_0)$, where $P$ is the vertical projection along the zero section $C\subset T^{*}\mathcal{F}$. By Lemma~\ref{dF}, the latter is given by
\begin{equation}\label{eq:res}
r(i^{*}\Omega_G^{\flat}(P(X_0)))=r(i^{*}\Omega_G^{\flat}(X_0)).
\end{equation}
The equality \eqref{eq:tosolve} says that 
\[
\Omega_G^{\flat}(X_t)=t\iota_{X_t}p^{*}d\tilde{\beta}+p^{*}\tilde{\beta},
\]
hence $\Omega_G^{\flat}(X_0)=p^{*}\tilde{\beta}$. Inserting this in the expression \eqref{eq:res}, we obtain
\[
\dt{\sigma_0}=r(i^{*}p^{*}\tilde{\beta})=r(\tilde{\beta})=\beta.
\]
\end{remark}

	\section{Implications for the coisotropic deformation problem}\label{sec:four}
	We discuss what Thm.~\ref{thm:main} tells us about the classical coisotropic deformation problem, which allows \emph{all} coisotropic deformations of $C$ instead of just those whose characteristic foliation is diffeomorphic with $\mathcal{F}$. 
	First, Thm.~\ref{thm:main} gives a partial unobstructedness result in this context, for which we provide some examples. Second, Thm.~\ref{thm:main} implies that the Kuranishi map of the $L_{\infty}$-algebra governing coisotropic deformations of $C$ should vanish on elements of $\ker(d_{\nu})$. We confirm that this is the case, by establishing a canonical description for the Kuranishi map in terms of the transverse differentiation map $d_{\nu}$.

	\subsection{A partial unobstructedness result}
	\vspace{0.1cm}
	\noindent
	
	Let $C\subset(M,\omega)$ be a compact coisotropic submanifold with characteristic foliation $\mathcal{F}$. As recalled in \S\ref{subsub:coiso}, first order deformations of $C$ --as a coisotropic submanifold-- are closed foliated one-forms on $\mathcal{F}$, and these are generally obstructed. It is however clear geometrically that a first order deformation is unobstructed if it admits a closed extension, see \cite[Rem.~4.6]{equivalences} or the proof of Lemma~\ref{lem:closedext}. Let us state this result for future reference.
	
	\begin{lemma}\label{lem:closed-extension}
		Let $C\subset(M,\omega)$ be a compact coisotropic submanifold. Any first order deformation $\beta\in\Omega^{1}(\mathcal{F})$ of $C$ admitting a closed extension $\widetilde{\beta}\in\Omega^{1}(C)$ is unobstructed.
	\end{lemma}
	
	By Thm.~\ref{thm:main} and Def.~\ref{def:firstorderdefs}, we obtain the following partial unobstructedness result, which extends Lemma~\ref{lem:closed-extension}.
	
	\begin{cor}\label{cor:partialunobs}
		Let $C\subset(M,\omega)$ be a compact coisotropic submanifold. Any first order deformation $\beta\in\Omega^{1}(\mathcal{F})$ of $C$ whose cohomology class $[\beta]\in H^{1}(\mathcal{F})$ lies in the kernel of the transverse differentiation map $d_{\nu}:H^{1}(\mathcal{F})\rightarrow H^{1}(\mathcal{F};N^{*}\mathcal{F})$ is unobstructed.
	\end{cor}

	The following example illustrates how we can use Cor.~\ref{cor:partialunobs} to find unobstructed first order deformations which cannot be detected with Lemma~\ref{lem:closed-extension}.

	\begin{ex}\label{ex:contactextensions}
		Let $C^{2n-q}\subset(M^{2n},\omega)$ be a compact coisotropic submanifold of $q$-contact type, as introduced in Ex. \ref{ex:contact}. So there exist $\alpha_1,\ldots,\alpha_q\in\Omega^{1}(C)$ such that:
		\begin{enumerate}
			\item $d\alpha_i=\omega_C$ for $i=1,\ldots,q$.
			\item $\alpha_1\wedge\ldots\wedge\alpha_q\wedge\omega_C^{n-q}$ is nowhere zero on $C$.
		\end{enumerate}
		Assume moreover that $q<n$, to exclude the possibility that $C$ is Lagrangian. As before, we denote $r:\Omega^{1}(C)\rightarrow\Omega^{1}(\mathcal{F})$ the restriction map. We claim that the foliated one-forms $r(\alpha_1),\ldots,r(\alpha_q)$ are unobstructed first order deformations of $C$ by virtue of Cor.~\ref{cor:partialunobs}. We also claim that none of them admits a closed extension, showing that this unobstructedness result cannot be obtained via Lemma~\ref{lem:closed-extension}.
		
		First note that item (1) above implies that $r(\alpha_i)$ is leafwise closed, so it indeed defines a first order deformation of $C$. Moreover, the extension $\alpha_i\in\Omega^{1}(C)$ of $r(\alpha_i)\in\Omega^{1}(\mathcal{F})$ is such that $d\alpha_i=\omega_C\in\Omega^{2}_{bas}(C)$. Hence, it follows from Lemma~\ref{lem:extension} that $[r(\alpha_i)]\in H^{1}(\mathcal{F})$ lies in the kernel of $d_{\nu}$. By Cor.~\ref{cor:partialunobs}, the first order deformations $r(\alpha_i)$ of $C$ are unobstructed.
		
		We now argue that the $r(\alpha_i)$ do not admit closed extensions. The crucial observation is that the characteristic foliation $\mathcal{F}$ of $C$ is \textbf{taut}, i.e. there exists a Riemannian metric on $C$ such that all the leaves of $\mathcal{F}$ become minimal submanifolds.
		
		\vspace{0.2cm}
		\noindent
		\underline{Claim:} The characteristic foliation $\mathcal{F}$ is taut. 
		
		\vspace{0.1cm}
		\noindent
		To prove the claim, recall that by Rummler's criterion \cite[Thm.~10.5.9]{candelconlon}, a $q$-dimensional foliation $\mathcal{F}$ on $C$ is taut exactly when there exists $\gamma\in\Omega^{q}(C)$ such that:
		\begin{enumerate}[i)]
			\item $\gamma$ restricts to a volume form on each leaf of $\mathcal{F}$,
			\item $\gamma$ is $\mathcal{F}$-closed, i.e. $d\gamma(V_1,\ldots,V_{q+1})=0$ whenever at least $q$ of the arguments $V_1,\ldots,V_{q+1}$ are tangent to $\mathcal{F}$.
		\end{enumerate}
		In the $q$-contact case, we can take $\gamma:=\alpha_1\wedge\ldots\wedge\alpha_q\in\Omega^{q}(C)$. The requirement (2) above states that $\alpha_1\wedge\ldots\wedge\alpha_q$ is nowhere zero when restricted to $\ker\omega_C=T\mathcal{F}$. This means that $\alpha_1\wedge\ldots\wedge\alpha_q$ restricts to a volume form on each leaf of $\mathcal{F}$. Moreover, evaluating
		\[
		d(\alpha_1\wedge\ldots\wedge\alpha_q)=\sum_{i=1}^{q}(-1)^{i+1}\alpha_1\wedge\ldots\wedge\alpha_{i-1}\wedge\omega_C\wedge\alpha_{i+1}\wedge\ldots\wedge\alpha_q
		\]
		on vector fields $V_1,\ldots,V_{q+1}$, at least $q$ of which are tangent to $\mathcal{F}$, gives zero. Indeed, the factor $\omega_C$ gets paired with two vector fields, at least one of which belongs to $T\mathcal{F}=\ker\omega_C$.
		
		\vspace{0.2cm}
		The above claim ensures that the class $[\omega_C]\in H^{2}_{bas}(C)$ is non-trivial. Indeed, if $\mathcal{F}$ is a taut codimension $2m$ foliation on a compact oriented manifold $C$, such that $\mathcal{F}$ is defined by a closed two-form $\omega_C\in\Omega^{2}(C)$, then the cohomology classes $[\omega_{C}^{k}]\in H^{2k}_{bas}(C)$ for $k=1,\ldots,m$ are non-trivial \cite[Thm.~4.33]{tondeur}. In the $q$-contact setting, the manifold $C$ is indeed oriented because $\alpha_1\wedge\ldots\wedge\alpha_q\wedge\omega_C^{n-q}$ is a volume form on $C$.
		
		We can now show that $r(\alpha_i)$ does not admit a closed extension. Assume by contradiction that $\beta_i$ is a closed extension of $r(\alpha_i)$. Then $\alpha_i-\beta_i\in\Omega^{1}_{bas}(C)$, because for $V\in\Gamma(T\mathcal{F})$ we have $\iota_{V}(\alpha_i-\beta_i)=0$ and 
		\[
		\pounds_{V}(\alpha_i-\beta_i)=\iota_{V}d(\alpha_i-\beta_i)=\iota_{V}d\alpha_i=\iota_{V}\omega_C=0.
		\]
		It follows that the basic form $d(\alpha_i-\beta_i)=\omega_C$ has a basic primitive $\alpha_i-\beta_i$, i.e. the class $[\omega_C]\in H^{2}_{bas}(C)$ is trivial. This is a contradiction, so $r(\alpha_i)$ has no closed extension.
	\end{ex}

	Clearly, Lemma~\ref{lem:closed-extension} does not yield many unobstructed first order deformations in practice. It only concerns closed $\beta\in\Omega^{1}(\mathcal{F})$ whose cohomology class lies in the image of the restriction map $[r]:H^{1}(C)\rightarrow H^{1}(\mathcal{F})$, and the latter is a finite dimensional subspace of $H^{1}(\mathcal{F})$. 
	
	By contrast, the kernel of $d_{\nu}:H^{1}(\mathcal{F})\rightarrow H^{1}(\mathcal{F};N^{*}\mathcal{F})$ is infinite dimensional in general\footnote{Note however that in the integral coisotropic setting considered by Ruan \cite{ruan}, this space is always finite dimensional. Indeed, if $(C,\mathcal{F})$ is integral then $\ker(d_{\nu})$ reduces to the space of flat sections of $(\mathcal{H}^{1},\nabla)$, see Ex.~\ref{ex:integral}. The latter is a finite dimensional vector space.}. 
	Hence, Cor.~\ref{cor:partialunobs} has the potential to generate considerably more unobstructed first order deformations. The following is an example in which $\ker(d_{\nu})$ is infinite dimensional.
	
	\begin{ex}\label{ex:inf}
		Consider $\mathbb{T}^{3}$ endowed with the presymplectic form $d\theta_2\wedge(d\theta_3-\cos(\theta_2)d\theta_1)$. The associated foliation is given by 
		\[
		T\mathcal{F}:=\text{Span}\{\partial_{\theta_1}+\cos(\theta_2)\partial_{\theta_3}\}.
		\]
		Hence, for every fixed value of $\theta_2$, we get either an $S^{1}$-fibration or a Kronecker foliation on $(\mathbb{T}^{2},\theta_1,\theta_3)$, depending on whether $\cos(\theta_2)$ is rational or irrational. The conormal bundle is $N^{*}\mathcal{F}=\text{Span}\{d\theta_2,d\theta_3-\cos(\theta_2)d\theta_1\}$, so we can use $\{\overline{d\theta_1}\}$ as a frame for $T^{*}\mathcal{F}=T^{*}\mathbb{T}^{3}/N^{*}\mathcal{F}$.
		Embedding this presymplectic manifold in its Gotay local model, it becomes a coisotropic submanifold. We claim that $\ker(d_{\nu})$ is infinite dimensional in this example.
		
		The map $d_{\nu}:H^{1}(\mathcal{F})\rightarrow H^{1}(\mathcal{F};N^{*}\mathcal{F})$ is given by 
		\begin{equation}\label{eq:transdif}
			d_{\nu}[g\overline{d\theta_1}]=\left[-\frac{\partial g}{\partial\theta_2}\overline{d\theta_1}\otimes d\theta_2-\frac{\partial g}{\partial\theta_3}\overline{d\theta_1}\otimes(d\theta_3-\cos(\theta_2)d\theta_1)\right].
		\end{equation}
		To see which elements are trivial in $H^{1}(\mathcal{F};N^{*}\mathcal{F})$, note that for $k,l\in C^{\infty}(\mathbb{T}^{3})$ we have
		\begin{equation}\label{eq:exact1}
			\Big\langle d_{\nabla^{*}}\big(kd\theta_2+l(d\theta_3-\cos(\theta_2)d\theta_1)\big)\big(\partial_{\theta_1}+\cos(\theta_2)\partial_{\theta_3}\big),\partial_{\theta_2}\Big\rangle
			=\frac{\partial k}{\partial\theta_1}+\cos(\theta_2)\frac{\partial k}{\partial\theta_3}-\sin(\theta_2)l
		\end{equation}
		and
		\begin{equation}\label{eq:exact2}
			\Big\langle d_{\nabla^{*}}\big(kd\theta_2+l(d\theta_3-\cos(\theta_2)d\theta_1)\big)\big(\partial_{\theta_1}+\cos(\theta_2)\partial_{\theta_3}\big),\partial_{\theta_3}\Big\rangle
			=\frac{\partial l}{\partial\theta_1}+\cos(\theta_2)\frac{\partial l}{\partial\theta_3}.
		\end{equation}
		We will now focus on foliated one-forms of the type $g(\theta_2)\overline{d\theta_1}$.
		If such a class $[g(\theta_2)\overline{d\theta_1}]$ lies in the kernel of $d_{\nu}:H^{1}(\mathcal{F})\rightarrow H^{1}(\mathcal{F};N^{*}\mathcal{F})$, then \eqref{eq:transdif}, \eqref{eq:exact1} and \eqref{eq:exact2} give 
		\begin{equation}\label{eq:system}
			\begin{dcases}
				\frac{\partial k}{\partial\theta_1}+\cos(\theta_2)\frac{\partial k}{\partial\theta_3}-\sin(\theta_2)l=-g'(\theta_2)\\
				\frac{\partial l}{\partial\theta_1}+\cos(\theta_2)\frac{\partial l}{\partial\theta_3}=0
			\end{dcases},
		\end{equation}
		for some $k,l\in C^{\infty}(\mathbb{T}^{3})$. For a fixed value $\theta_2=\theta_2^{0}$, the second equation in \eqref{eq:system} says that $l(\theta_1,\theta_2^{0},\theta_3)$ is a basic function for the foliated manifold $\big(\mathbb{T}^{2},\text{Span}\{\partial_{\theta_1}+\cos(\theta_2^{0})\partial_{\theta_3}\}\big)$. This implies that $l(\theta_1,\theta_2^{0},\theta_3)$ is constant when $\cos(\theta_2^{0})$ is irrational. Since the set of $\theta_2^{0}\in S^{1}$ for which $\cos(\theta_2^{0})$ is irrational is dense in $S^{1}$, it follows that the system of equalities
		\[
		\begin{dcases}
			\frac{\partial}{\partial\theta_1}l(\theta_1,\theta_2,\theta_3)=0\\
			\frac{\partial}{\partial\theta_3}l(\theta_1,\theta_2,\theta_3)=0
		\end{dcases}
		\]
		holds on a dense subset of $\mathbb{T}^{3}$, hence on all of $\mathbb{T}^{3}$. It follows that $l=l(\theta_2)$. The first equation in \eqref{eq:system} then implies that
		\[
		\frac{\partial k}{\partial\theta_1}+\cos(\theta_2)\frac{\partial k}{\partial\theta_3}=\sin(\theta_2)l(\theta_2)-g'(\theta_2),
		\]
		so integrating around tori $S^{1}\times\{\theta_2\}\times S^1$ we see that both sides must be zero. In conclusion,
		\[
		[g(\theta_2)\overline{d\theta_1}]\in\ker(d_{\nu})\ \Leftrightarrow\ g'(\theta_2)=l(\theta_2)\sin(\theta_2)\hspace{0.5cm}\text{for some}\ l\in C^{\infty}(S^1).
		\]
		
		Let us now consider the subspace 
		\begin{equation}\label{eq:subspace}
			\text{Span}\big\{[\cos^{n}(\theta_2)\overline{d\theta_1}]: n\geq 0\big\} \subset\ker(d_{\nu}).
		\end{equation}
		We claim that it is infinite dimensional. If not, then there would exist $m\in\mathbb{N}$ such that
		\[
		[\overline{d\theta_1}],[\cos(\theta_2)\overline{d\theta_1}],\ldots,[\cos^{m}(\theta_2)\overline{d\theta_1}]
		\]
		are linearly dependent. So there exist $c_0,\ldots,c_m\in\mathbb{R}$, not all zero, such that
		\[
		\sum_{i=0}^{m}c_i\cos^{i}(\theta_2)=\frac{\partial h}{\partial\theta_1}+\cos(\theta_2)\frac{\partial h}{\partial\theta_3}
		\]
		for some $h\in C^{\infty}(\mathbb{T}^{3})$. Integrating over the tori $S^{1}\times\{\theta_2\}\times S^1$, this implies that
		\[
		\sum_{i=0}^{m}c_i\cos^{i}(\theta_2)=0.
		\]
		But the polynomial $\sum_{i=0}^{m}c_ix^{i}$ has at most $m$ roots, whereas $\cos(\theta_2)$ takes on infinitely many values as $\theta_2$ ranges over $S^{1}$. Hence, we reach a contradiction, and therefore the subspace \eqref{eq:subspace} is infinite dimensional. The same then holds for $\ker(d_{\nu})$.
	\end{ex}
	
	We finish this subsection by highlighting a particular case of Cor.~\ref{cor:partialunobs}, which is important in the spirit of this note. Given a compact coisotropic submanifold 
	$C\subset(M,\omega)$, we noted in \S\ref{subsub:fol} that the presymplectic form $\omega_C$ induces an isomorphism $H^{1}(\mathcal{F};N^{*}\mathcal{F})\cong H^{1}(\mathcal{F};N\mathcal{F})$. Therefore, Cor.~\ref{cor:partialunobs} implies that the coisotropic deformation problem of $C$ is unobstructed when $H^{1}(\mathcal{F};N\mathcal{F})$ vanishes. The latter condition is the infinitesimal requirement for rigidity of the foliation $\mathcal{F}$, see Def.~\ref{def:inffol}.
	
	\begin{cor}\label{cor:rigidunobs}
		Let $C\subset(M,\omega)$ be a compact coisotropic submanifold with characteristic foliation $\mathcal{F}$. If $H^{1}(\mathcal{F};N\mathcal{F})$ vanishes, then the deformation problem of $C$ is unobstructed.
	\end{cor}
	
	We now display a non-trivial example of this type.
	
	\begin{ex}\label{ex:rigid}
		We revisit certain suspension foliations on mapping tori which were shown to be rigid in \cite{anosov}. We find conditions under which such a foliation $\mathcal{F}$ is given by the kernel of a presymplectic form. Gotay's theorem then guarantees that the mapping torus in question arises as a coisotropic submanifold with characteristic foliation $\mathcal{F}$.
		
		Pick a matrix $A\in SL(n,\mathbb{Z})$, where $n\geq 2$, which is diagonalizable over $\mathbb{R}$ with positive eigenvalues. Denote the eigenvalues by
		\[
		\mu_1,\ldots,\mu_p,\lambda_1,\ldots,\lambda_q,
		\]  
		where $p+q=n$. We view $A$ as a diffeomorphism of the torus $\mathbb{T}^{n}$. Pick independent linear vector fields $X_1,\ldots,X_p,Y_1,\ldots,Y_q\in\mathfrak{X}(\mathbb{T}^{n})$ such that\footnote{Put differently, $X_j$ (resp. $Y_k$) is a vector field whose coefficients are constant, given by the components of an eigenvector of $A$ for the eigenvalue $\mu_j$ (resp. $\lambda_k$).}
		\[
		\begin{cases}
			A_{*}X_j=\mu_jX_j,\\
			A_{*}Y_k=\lambda_k Y_k.
		\end{cases}
		\]
		The foliation $\text{Span}\{X_1,\ldots,X_p\}$ on $\mathbb{T}^{n}$ is invariant under $A$, hence the product foliation $\text{Span}\{X_1,\ldots,X_p,\partial_t\}$ on $\mathbb{T}^{n}\times\mathbb{R}$ descends to a foliation $\mathcal{F}$ on the mapping torus 
		\begin{equation}\label{eq:mappingtorus}
			\mathbb{T}_{A}:=\frac{\mathbb{T}^{n}\times\mathbb{R}}{(\theta,t)\sim(A(\theta),t+1)}.
		\end{equation}
		Assume moreover that the matrix $A$ satisfies the following two conditions:
		\begin{enumerate}
			\item The eigenvalues $\lambda_k$ and the quotients $\lambda_k/\mu_j$ are different from $1$.
			\item There is a basis of $\mathbb{R}^{n}$ given by eigenvectors $v_1,\ldots,v_p,w_1,\ldots,w_q$ of $A$ (corresponding respectively to the eigenvalues $\mu_1,\ldots,\mu_p,\lambda_1,\ldots,\lambda_q$) with the property that for any $i=1,\ldots,p$, the coordinates $v_i^{1},\ldots,v_i^{n}$ of $v_i$ are linearly independent over $\mathbb{Q}$.
		\end{enumerate}
		Under these assumptions, also the eigenvalues $\mu_1,\ldots,\mu_p$ are different from $1$, and therefore $A$ is an Anosov diffeomorphism of $\mathbb{T}^{n}$. If the eigenvalues of $A$ are all different, then condition $(2)$ is satisfied exactly when the characteristic polynomial of $A$ is irreducible over $\mathbb{Q}$.
		
		The main result of \cite{anosov} is that, under conditions $(1)$ and $(2)$, the foliation $\mathcal{F}$ is rigid. The proof proceeds by showing that the Bott complex $\big(\Omega^{\bullet}(\mathcal{F};N\mathcal{F}),d_{\nabla}\big)$ admits ``tame'' homotopy operators; this implies in particular that $H^{1}(\mathcal{F};N\mathcal{F})$ vanishes. 
		To realize $(\mathbb{T}_{A},\mathcal{F})$ as an instance of Cor.~\ref{cor:rigidunobs}, we have to find out when $\mathcal{F}$ is defined by a presymplectic form.
		
		\vspace{0.2cm}
		\noindent
		\underline{Claim:} The foliation $\mathcal{F}$ is defined by a presymplectic form on $\mathbb{T}_{A}$ exactly when the multiset $\{\lambda_1,\ldots,\lambda_q\}$ has the following property: if $\xi$ occurs with multiplicity $m$, then also $1/\xi$ occurs with multiplicity $m$.
		
		\vspace{0.1cm}
		\noindent
		Recall that all eigenvalues are assumed to be positive and different from $1$. Hence, the property mentioned in the claim implies in particular that $q$ is even, as it should be.
		To prove the claim, we denote by $\{\alpha_1,\ldots,\alpha_p,\beta_1,\ldots,\beta_q\}$ the frame of $T^{*}\mathbb{T}^{n}$ dual to the frame $\{X_1,\ldots,X_p,Y_1,\ldots,Y_q\}$ of $T\mathbb{T}^{n}$. This way, we obtain a frame
		\begin{equation}\label{eq:frame}
			\big\{dt,\mu_1^{-t}\alpha_1,\ldots,\mu_p^{-t}\alpha_p,\lambda_1^{-t}\beta_1,\ldots,\lambda_q^{-t}\beta_q\big\}
		\end{equation}
		for $T^{*}\mathbb{T}_{A}$. These are indeed $1$-forms on $\mathbb{T}^{n}\times\mathbb{R}$ invariant under the identification in \eqref{eq:mappingtorus}. 
		\begin{enumerate}[i)]
			\item For the forward implication, assume that $\omega$ is a presymplectic form on $\mathbb{T}_{A}$ with kernel $T\mathcal{F}$. Expressing $\omega$ in the frame \eqref{eq:frame} gives an equation of the form
			\[
			\omega=\sum_{i<j}f_{ij}(\theta,t)\lambda_i^{-t}\lambda_j^{-t}\beta_i\wedge\beta_j,
			\]
			where $f_{ij}$ satisfies
			\begin{equation}\label{eq:fij}
				f_{ij}\big(A(\theta),t+1\big)=f_{ij}(\theta,t).
			\end{equation}
			
			We now assert that, if $f_{ij}$ is not identically zero, then $\lambda_i\lambda_j=1$. To prove this, note that since $\pounds_{\partial_t}\omega$ vanishes, we have 
			\[
			f_{ij}\big(A(\theta),t\big)\lambda_i^{-t}\lambda_j^{-t}=f_{ij}\big(A(\theta),t+1\big)\lambda_i^{-t-1}\lambda_j^{-t-1}=f_{ij}(\theta,t)\lambda_i^{-t-1}\lambda_j^{-t-1},
			\]
			where we also used \eqref{eq:fij}. This implies that
			\begin{equation}\label{eq:rec2}
				\lambda_i\lambda_jf_{ij}\big(A(\theta),t)=f_{ij}(\theta,t).
			\end{equation}
			Now take a rational point $(\theta_1,\ldots,\theta_n)\in\mathbb{Q}^{n}/\mathbb{Z}^{n}$. Such a point is periodic for $A$. Indeed, it can be written as $(p_1/q,\ldots,p_n/q)$ for some $0\leq p_j<q$, and applying iterates of $A$ to it yields values in the finite set
			\[
			\big\{(q_1/q,\ldots,q_n/q):\ 0\leq q_j<q\big\}.
			\]
			Hence, the equality \eqref{eq:rec2} implies that for any point $(\theta_1,\ldots,\theta_n,t)\in \mathbb{Q}^{n}/\mathbb{Z}^{n}\times\mathbb{R}$, there exists an integer $k\geq 1$ such that
			\begin{equation}\label{eq:power}
				(\lambda_i\lambda_j)^{k}f_{ij}(\theta,t)=f_{ij}(\theta,t).
			\end{equation}
			If $f_{ij}$ is not identically zero, then there exists a point $(\theta_1,\ldots,\theta_n,t)\in \mathbb{Q}^{n}/\mathbb{Z}^{n}\times\mathbb{R}$ where $f_{ij}$ does not vanish. Hence, the equality \eqref{eq:power}, along with the fact that the eigenvalues are positive, yields that $\lambda_i\lambda_j=1$. The assertion is proved.
			
			This shows that, if $\xi$ occurs in the list $\{\lambda_1,\ldots,\lambda_q\}$, then also $1/\xi$ occurs in that list. For otherwise, there would exist $j\in\{1,\ldots,q\}$ so that the vector field $\xi^{t} Y_j\in\mathfrak{X}(\mathbb{T}_{A})$ lies in the kernel of $\omega$, while also being transverse to $\mathcal{F}$. To show that the multiplicities of $\xi$ and $1/\xi$ agree, we can assume that they are given by $k$ and $l$ respectively, and that
			\[
			\lambda_1=\cdots=\lambda_k=\xi\hspace{1cm}\text{and}\hspace{1cm}\lambda_{k+1}=\cdots=\lambda_{k+l}=1/\xi.	
			\]
			We know that the presymplectic form $\omega$ defines an isomorphism
			\[
			\omega^{\flat}:\text{Span}\big\{\lambda_1^{t}Y_1,\ldots,\lambda_q^{t}Y_q\big\}\overset{\sim}{\longrightarrow}\text{Span}\big\{\lambda_1^{-t}\beta_1,\ldots,\lambda_q^{-t}\beta_q\big\},
			\]
			and that at every point, it takes the subspace spanned by $\xi^{t}Y_1,\ldots,\xi^{t}Y_k$ into the subspace spanned by $(1/\xi)^{-t}\beta_{k+1},\ldots,(1/\xi)^{-t}\beta_{k+l}$. Hence, $l\geq k$. Because the argument is symmetric in $\xi$ and $1/\xi$, it follows that $l=k$. This proves the forward implication.
			\item For the backward implication, we have by assumption that $q=2k$. Moreover, after renumbering the elements of the list $\lambda_1,\ldots,\lambda_q$, we can assume that it is given by
			\[
			\lambda_1,\ldots,\lambda_k,\lambda_{k+1}=\lambda_1^{-1},\ldots,\lambda_{q}=\lambda_k^{-1}.
			\]
			We can then write down a well-defined two-form $\omega\in\Omega^{2}(\mathbb{T}_{A})$, given by
			\[
			\omega=\sum_{i=1}^{k}\beta_i\wedge\beta_{k+i}.
			\] 
			Note that the $\beta_i\in\Omega^{1}(\mathbb{T}^{n})$ are closed, since they are part of the frame dual to the commuting frame $\{X_1,\ldots,X_p,Y_1,\ldots,Y_q\}$ of $T\mathbb{T}^{n}$. This implies that $\omega$ is also closed. Clearly, the kernel of $\omega$ is given by $T\mathcal{F}$. This proves the backward implication.  
		\end{enumerate}
		
		Let us now give a concrete example of the type we just described. To find a matrix $A\in SL(n,\mathbb{Z})$ which satisfies the property mentioned in the claim above, it is useful to consider symplectic matrices. These always have determinant equal to $1$ \cite[Lemma~1.1.15]{mcduff}, and their characteristic polynomial has the property that, if $\xi$ is a root of multiplicity $m$, then also $1/\xi$ is a root of the same multiplicity \cite[Lemma~2.2.2]{mcduff}.
		Take for instance\footnote{A useful way to construct symplectic matrices is the following \cite[Example 17]{symplmatrix}. Let $X$ and $Y$ be two symmetric $n\times n$ matrices, with $X$ invertible. Then
			\[
			S:=\begin{pmatrix}
				X+YX^{-1}Y & YX^{-1}\\
				X^{-1}Y & X^{-1}
			\end{pmatrix}
			\]
			is a symplectic matrix. The matrix $A$ above is obtained by taking
			\[
			X=\begin{pmatrix}
				1 & 0\\
				0 & 1
			\end{pmatrix},\  Y=\begin{pmatrix}
				1 & 1\\
				1 & 0
			\end{pmatrix}.
			\]
		}
		\[
		A:=\begin{pmatrix}
			3 & 1 & 1 & 1\\
			1 & 2 & 1 & 0\\
			1 & 1 & 1 & 0\\
			1 & 0 & 0 & 1
		\end{pmatrix}\in SL(4,\mathbb{Z}).
		\] 
		The eigenvalues are
		\[
		\mu\approx 4.39, \mu^{-1}, \lambda\approx 1.84, \lambda^{-1},
		\]
		in particular they are all distinct. The characteristic polynomial is 
		\[
		X^{4}-7X^{3}+13X^{2}-7X+1.
		\]
		Its image in $(\mathbb{Z}/2\mathbb{Z})[X]$ is the irreducible polynomial $X^{4}+X^{3}+X^{2}+X+1$, hence the characteristic polynomial is irreducible over $\mathbb{Z}$ and therefore over $\mathbb{Q}$. It follows that the assumptions $(1)$ and $(2)$ from \cite{anosov} are satisfied. Hence, taking for instance linear vector fields $X_1,X_2$ on $\mathbb{T}^{4}$ satisfying
		\[
		\begin{cases}
			A_{*}X_1=\mu X_1\\
			A_{*}X_2=\mu^{-1} X_2
		\end{cases},
		\]
		the associated suspension foliation $\mathcal{F}$ on the mapping torus $\mathbb{T}_{A}$ is rigid. Now pick a presymplectic form $\omega$ on $\mathbb{T}_{A}$ with kernel $T\mathcal{F}$ and embed $(\mathbb{T}_{A},\omega)$ coisotropically into its Gotay local model. The result is a coisotropic submanifold whose deformation problem is unobstructed.
	\end{ex}
	
	\begin{remark}
		The fact that the deformation problem of the coisotropic submanifolds from Ex.~\ref{ex:rigid} is unobstructed also follows from Lemma~\ref{lem:closed-extension}, because the restriction  
		\begin{equation}\label{eq:surj}
			H^{1}(\mathbb{T}_{A})\rightarrow H^{1}(\mathcal{F})
		\end{equation}
		is surjective. To see why, we recall that leafwise cohomology of suspension foliations can be computed explicitly using a Mayer-Vietoris argument, see \cite[\S 1.3.3]{foliations}.
		Denoting by $\mathcal{F}_{0}$ the foliation on $\mathbb{T}^{n}$ given by
		\[
		T\mathcal{F}_{0}=\text{Span}\{X_1,\ldots,X_p\},
		\]
		one gets
		\begin{equation}\label{eq:folcohomology}
			H^{1}(\mathcal{F})=\ker\Big(\text{Id}-[A^{*}]:H^{1}(\mathcal{F}_{0})\rightarrow H^{1}(\mathcal{F}_{0})\Big)\oplus\frac{H^{0}(\mathcal{F}_{0})}{\text{im}\Big(\text{Id}-[A^{*}]:H^{0}(\mathcal{F}_{0})\rightarrow H^{0}(\mathcal{F}_{0})\Big)}.
		\end{equation}
		The assumptions $(1)$ and $(2)$ from \cite{anosov} (recalled in Ex.~\ref{ex:rigid} above) ensure that $\mathcal{F}_{0}$ is a so-called Diophantine linear foliation on $\mathbb{T}^{n}$, and leafwise cohomology of such foliations was computed in \cite{arraut} (see also \cite[Thm.~1.3.7]{foliations}). Continuing in the notation of Ex.~ \ref{ex:rigid}, one has
		\[
		H^{1}(\mathcal{F}_{0})=\mathbb{R}[r(\alpha_1)]\oplus\cdots\oplus\mathbb{R}[r(\alpha_p)],
		\]
		where $r:\Omega^{1}(\mathbb{T}^{n})\rightarrow\Omega^{1}(\mathcal{F}_{0})$ is the restriction map. Because $A^{*}\alpha_j=\mu_j\alpha_j$ and $\mu_1,\ldots,\mu_q$ are different from $1$, it follows that the map $\text{Id}-[A^{*}]:H^{1}(\mathcal{F}_{0})\rightarrow H^{1}(\mathcal{F}_{0})$ is injective, hence the first summand in \eqref{eq:folcohomology} vanishes. Because the leaves of $\mathcal{F}_{0}$ are dense, $H^{0}(\mathcal{F}_{0})$ consists of just the constant functions on $\mathbb{T}^{n}$. Hence, the expression \eqref{eq:folcohomology} shows that $H^{1}(\mathcal{F})\cong\mathbb{R}$.
		
		A representative is given by the restriction of $dt$ to the leaves of $\mathcal{F}$. To see that this is a non-exact foliated one-form, let $L$ be the leaf of $\mathcal{F}$ through a rational point $\overline{(\theta,t)}\in\mathbb{T}_{A}$, i.e. $\theta\in\mathbb{Q}^{n}/\mathbb{Z}^{n}$. The flow line of $\partial_t$ through $\overline{(\theta,t)}$ stays inside $L$ and is closed, because $\theta$ is a periodic point for $A$ (see Ex. \ref{ex:rigid} i)). By Stokes' theorem, the restriction of $dt$ to the leaf $L$ cannot be exact. Since $H^{1}(\mathcal{F})$ is spanned by the class of $dt$, it is now clear that the restriction map \eqref{eq:surj} is indeed surjective.
	\end{remark}

	\subsection{The Kuranishi criterion}
	\vspace{0.1cm}
	\noindent
	
	In this subsection, we recall that the deformation problem of a coisotropic submanifold $C\subset(M,\omega)$ is governed by a suitable $L_{\infty}[1]$-algebra. The latter comes with a tool that allows one to detect obstructed first order deformations, called the Kuranishi criterion. We check that the partial unobstructedness result from Cor.~\ref{cor:partialunobs} is consistent with this criterion.
	
	\begin{defi}
		\begin{enumerate}[i)]
			\item An $\mathbf{L_{\infty}[1]}$\textbf{-algebra} is a $\mathbb{Z}$-graded vector space $W$, equipped with a collection of graded symmetric multibrackets $(\lambda_k:W^{\otimes k}\rightarrow W)_{k\geq 1}$ of degree $1$ which satisfy a collection of relations \cite{lada} called higher Jacobi identities.
			\item The \textbf{Maurer-Cartan series} of an element $w\in W$ of degree $0$ is the infinite sum
			\[
			MC(w):=\sum_{k=1}^{\infty}\frac{1}{k!}\lambda_k(w^{\otimes k}).
			\]
			A \textbf{Maurer-Cartan element} is a degree zero element $w\in W$ for which the Maurer-Cartan series converges to zero.
			
			\item We say that an $L_{\infty}[1]$-algebra $\big(W,\{\lambda_k\}\big)$ governs a certain deformation problem if small deformations correspond with Maurer-Cartan elements of $\big(W,\{\lambda_k\}\big)$.
		\end{enumerate}
	\end{defi}
	
	\begin{remark}
		The higher Jacobi identities mentioned above imply in particular that $\lambda_1$ is a differential, which gives rise to cohomology groups $H(W)$. Moreover, the binary bracket $\lambda_2$ descends to the cohomology $H(W)$, and this fact which will be essential in the sequel.
	\end{remark}
	
	Oh and Park showed in \cite{oh-park} that every coisotropic submanifold $C\subset(M,\omega)$ has an attached $L_{\infty}[1]$-algebra structure. We recall an elegant description of it in terms of derived brackets, which is due to Cattaneo and Felder \cite{cattaneo-felder}. The $L_{\infty}[1]$-algebra associated with $C\subset(M,\omega)$ is obtained using the Gotay local model $(U,\Omega_G)$, for a choice of complement $G$ to the characteristic foliation $T\mathcal{F}$ on $C$. Its multibrackets $\{\lambda_k\}_{k\geq 1}$ are defined on the graded vector space $\Gamma(\wedge T^{*}\mathcal{F})[1]$, by the formula\footnote{The graded vector space $\Gamma(\wedge T^{*}\mathcal{F})[1]$ is just $\Gamma(\wedge T^{*}\mathcal{F})$ up to degree shift. Namely, a degree $k$ element of $\Gamma(\wedge T^{*}\mathcal{F})[1]$ lives in $\Gamma(\wedge^{k+1} T^{*}\mathcal{F})$.}
	\begin{equation}\label{multibracket}
		\lambda_{k}(\xi_1,\ldots,\xi_k)=P\left([[\cdots[\Pi_G,\xi_1],\cdots,\xi_{k-1}],\xi_k]\right).
	\end{equation}
	Here $\Pi_G:=-\Omega_G^{-1}$ denotes the Poisson structure corresponding with $\Omega_G$, and we consider the sections $\xi_i\in\Gamma(\wedge T^{*}\mathcal{F})$ as vertical fiberwise constant multivector fields on $T^{*}\mathcal{F}$ via the correspondence in Rem.~\ref{rem:corres}. The bracket $[-,-]$ appearing in \eqref{multibracket} is the Schouten-Nijenhuis bracket of multivector fields, and the map $P$ acts on multivector fields by first restricting them to $C$ and then taking the vertical projection $\Gamma(\wedge^{\bullet}T(T^{*}\mathcal{F})|_{C})\rightarrow\Gamma(\wedge^{\bullet}T^{*}\mathcal{F})$ coming from the splitting $T(T^{*}\mathcal{F})|_{C}=TC\oplus T^{*}\mathcal{F}$.
	While the work of Oh-Park addresses formal deformations of $C$, it was shown by Sch\"{a}tz-Zambon in \cite{fiberwise} that the $L_{\infty}[1]$-algebra $\big(\Gamma(\wedge T^{*}\mathcal{F})[1],\{\lambda_k\}\big)$ actually governs the smooth coisotropic deformation problem of $C$.

	\begin{prop}[\cite{fiberwise}]\label{prop:fib}
		Let $C\subset(M,\omega)$ be a coisotropic submanifold with Gotay local model $(U,\Omega_G)$. For any $\alpha\in\Gamma(T^{*}\mathcal{F})$ whose graph is contained in $U$, the Maurer-Cartan series $MC(-\alpha)$ is pointwise convergent. For such $\alpha$, the following are equivalent:
		\begin{enumerate}
			\item $\text{Graph}(\alpha)$ is coisotropic in $(U,\Omega_G)$,
			\item The Maurer-Cartan series $MC(-\alpha)$ converges to zero.
		\end{enumerate} 
	\end{prop}
	The differential $\lambda_1$ is just the leafwise de Rham differential $d_{\mathcal{F}}$ up to sign, see \cite[Proof of Prop.~3.5]{fiberwise}. Hence, the linearization of the Maurer-Cartan equation is just the closedness condition with respect to $d_{\mathcal{F}}$. This is another way to see that first order deformations of $C$ are closed foliated one-forms $\beta\in\Omega^{1}(\mathcal{F})$, a result which we already proved in Lemma~\ref{lem:infcoiso}.
	
	If $\beta\in\Omega^{1}(\mathcal{F})$ is an unobstructed first order deformation of $C$, then there exists a smooth path $\alpha_t$ of coisotropic sections of $(U,\Omega_G)$ starting at $\alpha_0=0$, satisfying $\dt{\alpha_0}=\beta$. By Prop.~\ref{prop:fib} above, $-\alpha_t$ is a Maurer-Cartan element of $\big(\Gamma(\wedge T^{*}\mathcal{F})[1],\{\lambda_k\}\big)$ for all $t$. Differentiating twice the equality
	\[
	MC(-\alpha_t)=0
	\]
	and evaluating at time $t=0$, it follows that $\lambda_2(\beta,\beta)$ is exact in $\big(\Omega(\mathcal{F}),d_{\mathcal{F}}\big)$. This is the \textbf{Kuranishi criterion} for unobstructedness of a first order deformation \cite[Thm.~11.4]{oh-park}.
	
	\begin{prop}
		The \textbf{Kuranishi map} of the $L_{\infty}[1]$-algebra $\big(\Gamma(\wedge T^{*}\mathcal{F})[1],\{\lambda_k\}\big)$ is 
		\[
		Kr:H^{1}(\mathcal{F})\rightarrow H^{2}(\mathcal{F}):[\beta]\mapsto[\lambda_2(\beta,\beta)].
		\]
		If $\beta\in\Omega^{1}(\mathcal{F})$ is an unobstructed first order deformation of $C$, then $Kr([\beta])$ vanishes.
	\end{prop}
	
	We know from Cor.~\ref{cor:partialunobs} that first order deformations $\beta$ whose cohomology class lies in the kernel of $d_{\nu}:H^{1}(\mathcal{F})\rightarrow  H^{1}(\mathcal{F};N^{*}\mathcal{F})$ are unobstructed. Hence, we should have that $\ker(d_{\nu})\subset\ker(Kr)$. In the following, we double-check that this inclusion holds.
	
	We will actually prove a more interesting result. It is clear from the expression \eqref{multibracket} that the multibrackets $\lambda_k$ of the $L_{\infty}[1]$-algebra $\big(\Gamma(\wedge T^{*}\mathcal{F})[1],\{\lambda_k\}\big)$ depend on the choice of complement $G$ to $T\mathcal{F}$. By contrast, we will show that the Kuranishi map admits a canonical description in terms of the presymplectic form $\omega_C$ and the map $d_{\nu}$. This result highlights once more the role played by the transverse differentiation map $d_{\nu}$ in the coisotropic deformation problem, and it will imply in particular that  $\ker(d_{\nu})\subset\ker(Kr)$.
	
	\begin{prop}\label{prop:binary}
		Let $C\subset(M,\omega)$ be a coisotropic submanifold. Choose a complement $G$ to the characteristic distribution $T\mathcal{F}$, and let $\left(\Gamma(\wedge T^{*}\mathcal{F})[1],\{\lambda_k\}\right)$ denote the $L_{\infty}[1]$-algebra associated with the corresponding Gotay local model. Its binary bracket satisfies
		\begin{equation}\label{eq:bin-br}
		\lambda_2(\alpha,\beta)=-\big\langle\!\big\langle\big(\text{Id}\otimes(\omega_C^{\flat})^{-1}\big)(\tau(d_{1,0}\alpha)),\tau(d_{1,0}\beta)\big\rangle\!\big\rangle,\hspace{0.5cm}\alpha,\beta\in\Omega^{1}(\mathcal{F}).
		\end{equation}
	\end{prop}
	
	In the statement of the proposition, we denoted by $\langle\!\langle\bullet,\bullet\rangle\!\rangle$ the duality pairing
	\begin{equation}\label{eq:pairing}
		\Omega^{k}(\mathcal{F};G)\times\Omega^{l}(\mathcal{F};G^{*})\rightarrow\Omega^{k+l}(\mathcal{F}):\big(\eta\otimes Y,\xi\otimes \gamma\big)\mapsto\langle\gamma,Y\rangle\eta\wedge\xi.
	\end{equation}
	By writing $\tau\circ d_{1,0}$, we used the bi-degree language introduced in \S\ref{sec:appendix_transdiff} of the Appendix.

	\begin{remark}
	The original paper \cite{oh-park} by Oh-Park contains a formula for $\lambda_2$ which looks somewhat similar to \eqref{eq:bin-br}. However, the Oh-Park formula depends not only on the choice of complement $G$ but also on a choice of coordinates, since it involves a non-canonical connection defined in a suitable chart. By contrast, the expression \eqref{eq:bin-br} is global, due to the fact that it uses the operator $d_{1,0}$ instead of the aforementioned local connection. It is easy to check that, when expressed in coordinates, our equation \eqref{eq:bin-br} reduces to the formula given by Oh-Park, hence it provides an invariant description for the latter.  
	\end{remark}

	\begin{proof}[Proof of Prop.~\ref{prop:binary}]
		Using the definition \eqref{multibracket} of $\lambda_2$ and the identification in Lemma \ref{dF} between sections of $T^{*}\mathcal{F}$ and vertical fiberwise constant vector fields on $U$, we have
		\[
		\lambda_2(\alpha,\beta)=P\Big[\Big[\Pi_G,\Pi_G^{\sharp}(p^{*}(j(\alpha)))\Big],\Pi_G^{\sharp}(p^{*}(j(\beta)))\Big].
		\]
		By Rem.~\ref{rem:corres}, the associated foliated two-form is
		\begin{align}\label{eq:twoform}
			&r\left(i^{*}\left(\wedge^{2}\Omega_G^{\flat}\left(P\Big[\Big[\Pi_G,\Pi_G^{\sharp}(p^{*}(j(\alpha)))\Big],\Pi_G^{\sharp}(p^{*}(j(\beta)))\Big]\right)\right)\right)\nonumber\\
			&\hspace{1cm}=r\left(i^{*}\left(\wedge^{2}\Omega_G^{\flat}\left(\Big[\Big[\Pi_G,\Pi_G^{\sharp}(p^{*}(j(\alpha)))\Big],\Pi_G^{\sharp}(p^{*}(j(\beta)))\Big]\right)\right)\right)\nonumber\\
			&\hspace{1cm}=-r\left(i^{*}\left(\wedge^{2}\Omega_G^{\flat}\Big[\wedge^{2}\Pi_G^{\sharp}(p^{*}(d(j(\alpha)))),\Pi_G^{\sharp}(p^{*}(j(\beta)))\Big]\right)\right).
		\end{align}
		In the first equality, we used that $\Omega_G(TC,T\mathcal{F})=0$, and the second equality uses that $\wedge\Pi_G^{\sharp}$ intertwines $[\Pi_G,\bullet]$ and $-d$ (see \cite[Lemma 2.1.3]{libermann}). To simplify the expression \eqref{eq:twoform}, note that for $\gamma_1,\gamma_2,\gamma_3\in\Omega^{1}(U)$ we have
		\begin{equation}\label{eq:id}
			\big[\wedge^{2}\Pi_G^{\sharp}(\gamma_1\wedge\gamma_2),\Pi_G^{\sharp}(\gamma_3)\big]=\wedge^{2}\Pi_{G}^{\sharp}\big[\gamma_1\wedge\gamma_2,\gamma_3\big]_{\Pi_G},
		\end{equation}
		where $[\bullet,\bullet]_{\Pi_G}$ denotes the Koszul bracket associated with $\Pi_G$ (see \cite[\S4.1]{koszul}). The identity \eqref{eq:id} is proved using the derivation rules for the Schouten bracket and the Koszul bracket, along with the fact that $\Pi_G^{\sharp}$ intertwines these brackets when applied to vector fields and one-forms, respectively.
		Consequently, the expression \eqref{eq:twoform} becomes
		\begin{equation*}
			-r\left(i^{*}\big[p^{*}(d(j(\alpha))),p^{*}(j(\beta))\big]_{\Pi_G}\right).
		\end{equation*}
		To simplify this further, we will first study the Poisson structure $\Pi_G$ along the zero section.
		
		\vspace{0.2cm}
		
		\underline{Claim 1:} The bivector field $Z:=\wedge^{2}p_{*}(\Pi_{G}|_{C})\in\Gamma(\wedge^{2}TC)$ actually lies in $\Gamma(\wedge^{2}G)$. It is characterized by the requirement that $Z^{\sharp}:G^{*}\rightarrow G$ equals $-(\omega_{C}^{\flat})^{-1}$.
		
		\vspace{0.1cm}
		\noindent
		We prove the claim, starting from the decomposition $T^{*}\mathcal{F}|_{C}=TC\oplus T^{*}\mathcal{F}=(T\mathcal{F}\oplus G)\oplus T^{*}\mathcal{F}$. As shown in the proof of \cite[Prop. 2.14]{b-coiso}, the symplectic form $\Omega_G$ decomposes as follows at points $x\in C$:
		\[
		\left(\Omega_G\right)_{x}=\begin{blockarray}{cccc}
			& T_{x}\mathcal{F} & G_{x} & T_{x}^{*}\mathcal{F} \\
			\begin{block}{c(ccc)}
				T_{x}\mathcal{F} & 0 & 0 & A \\
				G_{x} & 0 & B & 0 \\
				T_{x}^{*}\mathcal{F} & -A & 0 & 0 \\
			\end{block}
		\end{blockarray}~,
		\]
		for invertible matrices $A,B$. Note that $B^{T}$ represents the isomorphism $\omega_C^{\flat}:G\rightarrow G^{*}$. Consequently, the Poisson structure $\Pi_G$ decomposes as follows at points $x\in C$:
		\begin{equation}\label{eq:matrixpi}
			\left(\Pi_G\right)_{x}=-\left(\Omega_G\right)_{x}^{-1}=\begin{blockarray}{cccc}
				& T_{x}^{*}\mathcal{F} & G_{x}^{*} & T_{x}\mathcal{F} \\
				\begin{block}{c(ccc)}
					T_{x}^{*}\mathcal{F} & 0 & 0 & A^{-1} \\
					G_{x}^{*} & 0 & -B^{-1} & 0 \\
					T_{x}\mathcal{F} & -A^{-1} & 0 & 0 \\
				\end{block}
			\end{blockarray}~.
		\end{equation}
		Since the bivector $Z_{x}$ is represented by the $(2\times2)$ top-left block matrix in \eqref{eq:matrixpi}, this shows that $Z\in\Gamma(\wedge^{2}G)$ and  $Z^{\sharp}=(-B^{-1})^{T}=-(\omega_{C}^{\flat})^{-1}$ as maps $G^{*}\rightarrow G$. The proves the claim.
		
		\vspace{0.2cm}
		
		\underline{Claim 2:} For any $\gamma_1,\gamma_2\in\Omega^{1}(C)$, we have
		\[
		i^{*}[p^{*}\gamma_1,p^{*}\gamma_2]_{\Pi_G}=[\gamma_1,\gamma_2]_{Z}.
		\]
		\vspace{0.1cm}
		\noindent
		To prove Claim 2, we compute using Claim 1:
		\begin{align*}
			i^{*}[p^{*}\gamma_1,p^{*}\gamma_2]_{\Pi_G}&=i^{*}\left(\iota_{\Pi_G^{\sharp}(p^{*}\gamma_1)}p^{*}(d\gamma_2)-\iota_{\Pi_G^{\sharp}(p^{*}\gamma_2)}p^{*}(d\gamma_1)+d\big(\Pi_G(p^{*}\gamma_1,p^{*}\gamma_2)\big)\right)\\
			&=\iota_{Z^{\sharp}(\gamma_1)}d\gamma_2-\iota_{Z^{\sharp}(\gamma_2)}d\gamma_1+d\big(Z(\gamma_1,\gamma_2)\big)\\
			&=[\gamma_1,\gamma_2]_{Z}.
		\end{align*}
		
		\vspace{0.2cm}
		
		We now finish the proof of the proposition, by showing that
		\begin{equation}\label{eq:finalize}
			r\left(i^{*}\big[p^{*}(d(j(\alpha))),p^{*}(j(\beta))\big]_{\Pi_G}\right)=\big\langle\!\big\langle\big(\text{Id}\otimes(\omega_C^{\flat})^{-1}\big)(\tau(d_{1,0}\alpha)),\tau(d_{1,0}\beta)\big\rangle\!\big\rangle.
		\end{equation}
		Note that both sides are linear in $d(j(\alpha))$. This is clear for the left hand side, and for the right hand side this follows from the fact that $d_{1,0}\alpha$ is just the component of $d(j(\alpha))$ lying in $\Gamma(G^{*}\otimes T^{*}\mathcal{F})$. Hence, to prove the equality \eqref{eq:finalize}, we may assume that
		\begin{equation}\label{eq:assumption}
			d(j(\alpha))=\xi_1\wedge\xi_2+\xi\wedge\eta+\eta_1\wedge\eta_2\in\Gamma(\wedge^{2}T^{*}\mathcal{F})\oplus\Gamma(T^{*}\mathcal{F}\otimes G^{*})\oplus\Gamma(\wedge^{2}G^{*}).
		\end{equation}
		Using Claim 1 and Claim 2, the left hand side in \eqref{eq:finalize} is given by
		\begin{align*}
			&\hspace{0.7cm}r\left(i^{*}\big[p^{*}\xi_1\wedge p^{*}\xi_2+p^{*}\xi\wedge p^{*}\eta+p^{*}\eta_1\wedge p^{*}\eta_2,p^{*}(j(\beta))\big]_{\Pi_G}\right)\\
			&=r\left(i^{*}\left(p^{*}\xi_1\wedge\big[p^{*}\xi_2,p^{*}(j(\beta))\big]_{\Pi_G}\hspace{-0.15cm}+\big[p^{*}\xi_1,p^{*}(j(\beta))\big]_{\Pi_G}\hspace{-0.15cm}\wedge p^{*}\xi_2+p^{*}\xi\wedge\big[p^{*}\eta,p^{*}(j(\beta))\big]_{\Pi_G}\right)\right)\\
			&=\xi_1\wedge r\big([\xi_2,j(\beta)]_{Z}\big)+r\big([\xi_1,j(\beta)]_{Z}\big)\wedge\xi_2+\xi\wedge r\big([\eta,j(\beta)]_{Z}\big)\\
			&=\xi\wedge r\left(\iota_{Z^{\sharp}(\eta)}d(j(\beta))\right).
		\end{align*}
		On the other hand, from \eqref{eq:assumption} it immediately follows that
		\[
		\big(\text{Id}\otimes(\omega_C^{\flat})^{-1}\big)(\tau(d_{1,0}\alpha))=-\big(\text{Id}\otimes(\omega_C^{\flat})^{-1}\big)(\xi\wedge\eta)
		=-\xi\otimes(\omega_{C}^{\flat})^{-1}(\eta)
		=\xi\otimes Z^{\sharp}(\eta),
		\]
		hence the right hand side in \eqref{eq:finalize} is given by
		\begin{align*}
			\big\langle\!\big\langle\big(\text{Id}\otimes(\omega_C^{\flat})^{-1}\big)(\tau(d_{1,0}\alpha)),\tau(d_{1,0}\beta)\big\rangle\!\big\rangle&=\xi\wedge\big\langle\!\big\langle Z^{\sharp}(\eta),\tau(d_{1,0}\beta)\big\rangle\!\big\rangle\\
			&=\xi\wedge r\left(\iota_{Z^{\sharp}(\eta)}d(j(\beta))\right).
		\end{align*}
		This shows that the equality \eqref{eq:finalize} holds, and this in turn proves the proposition.
	\end{proof}
	
	\begin{remark}
		If the complement $G$ is involutive, then the proof of Prop. \ref{prop:binary} becomes more geometric. In that case, we have a symplectic foliation $(G,\omega_C)$ defining a Poisson structure $\Pi_{base}$ on $C$, and the projection $p:(U,\Pi_G)\rightarrow (C,\Pi_{base})$ is a Poisson map. This immediately implies Claim 1 in the proof above.
		It also implies Claim 2, because of the following. Recall that the since $p$ is a Poisson map, the pullback bundle  $p^{*}T^{*}C$ carries a natural Lie algebroid structure $(\rho,[\bullet,\bullet]_{p^{*}})$ determined by
		\[
		[p^{*}\gamma_1,p^{*}\gamma_2]_{p^{*}}=p^{*}[\gamma_1,\gamma_2]_{\Pi_{base}},\hspace{0.5cm}\rho(p^{*}\gamma)=\Pi_{G}^{\sharp}(p^{*}\gamma).
		\]
		When $p^{*}T^{*}C$ is endowed with this Lie algebroid structure and $T^{*}U$ carries its usual cotangent Lie algebroid structure $(\Pi_G^{\sharp},[\bullet,\bullet]_{\Pi_G})$, the natural map $p^{*}T^{*}C\rightarrow T^{*}U$ is a Lie algebroid morphism \cite[Prop. 1.10]{modular}. But since $p:U\rightarrow C$ is a submersion, the map $p^{*}T^{*}C\rightarrow T^{*}U$ is injective, hence $(p^{*}T^{*}C,\rho,[\bullet,\bullet]_{p^{*}})$ becomes a Lie subalgebroid of $(T^{*}U,\Pi_G^{\sharp},[\bullet,\bullet]_{\Pi_G})$. Compatibility of their Lie brackets implies Claim 2 in the proof above.	
	\end{remark}
	
	At last, we pass to cohomology. The pairing \eqref{eq:pairing} descends to 
	\[
	\langle\!\langle\bullet,\bullet\rangle\!\rangle:H^{k}(\mathcal{F};N\mathcal{F})\times H^{l}(\mathcal{F};N^{*}\mathcal{F})\rightarrow H^{k+l}(\mathcal{F}),
	\]
	the map $(-1)^{\bullet}(\tau\circ d_{1,0})$ descends to $d_{\nu}$ by Cor.~\ref{cor:transversediff}, and we have an isomorphism
	\[
	[\text{Id}\otimes(\omega_{C}^{\flat})^{-1}]:H^{\bullet}(\mathcal{F};N^{*}\mathcal{F})\rightarrow H^{\bullet}(\mathcal{F};N\mathcal{F}),
	\]
	see \S\ref{subsub:fol}. Hence, Prop.~\ref{prop:binary} immediately implies the following.
	
	\begin{cor}\label{cor:kur}
		The Kuranishi map of $\big(\Gamma(\wedge T^{*}\mathcal{F})[1],\{\lambda_k\}\big)$ is given canonically by
		\[
		Kr:H^{1}(\mathcal{F})\rightarrow H^{2}(\mathcal{F}):[\beta]\mapsto -\Big\langle\!\Big\langle \left(\big[\text{Id}\otimes(\omega_{C}^{\flat})^{-1}\big]\circ d_{\nu}\right)[\beta],d_{\nu}[\beta] \Big\rangle\!\Big\rangle.
		\]
		In particular, $\ker(d_{\nu})$ is contained in $\ker(Kr)$.
	\end{cor}

	\section{Appendix}
	
	\subsection{Proofs for results in \S\ref{subsubsec:canonical}}\label{sec:appendix_transdiff}
	\vspace{0.1cm}
	\noindent
	
	This subsection is devoted to the proofs of Prop.~\ref{prop:rephrase-canonical} and Lemma~\ref{lem:extension}. Both proofs will use some ingredients of the spectral sequence of the foliation $\mathcal{F}$, which we introduce now. 
	
	Recall that we fixed a complement $G$ to the characteristic distribution $T\mathcal{F}$ on $C$. 
	The decomposition $TC=T\mathcal{F}\oplus G$ induces a bi-grading on $\Omega(C)$, namely
	\[
	\Omega^{k}(C)=\bigoplus_{u+v=k}\Omega^{u,v}(C),
	\]
	where
	\[
	\Omega^{u,v}(C):=\Gamma(\wedge^{u}G^{*}\otimes\wedge^{v}T^{*}\mathcal{F}).
	\]
	With respect to this bi-grading, the de Rham differential splits into a sum of bihomogeneous components 
	\begin{equation}\label{eq:components}
		d=d_{0,1}+d_{1,0}+d_{2,-1},
	\end{equation}
	where the subscript $(i,j)$ indicates the bi-degree of the component $d_{i,j}$. Note that $d_{2,-1}$ vanishes exactly when the complement $G$ is involutive. We will use the following explicit formulae for $d_{0,1}$ and $d_{1,0}$, which can be found in \cite[Chapter 4]{tondeur}. If $\omega\in\Omega^{u,v}(C)$, then 
	
	\begin{align*}
		(d_{0,1}\omega)(Y_1,\ldots,Y_u;V_1,\ldots,V_{v+1})&=\sum_{i=1}^{v+1}(-1)^{u+i+1}V_i\big(\omega(Y_1,\ldots,Y_u;V_1,\ldots,\widehat{V_i},\ldots,V_{v+1})\big)\\
		&\hspace{-2.5cm}+\sum_{1\leq i<j\leq v+1}(-1)^{i+j+u}\omega(Y_1,\ldots,Y_u;[V_i,V_j],V_1,\ldots,\widehat{V_i},\ldots,\widehat{V_j},\ldots,V_{k+1})\\
		&\hspace{-2.5cm}+\sum_{\alpha=1}^{u}\sum_{j=1}^{v+1}(-1)^{u+\alpha+j}\omega\big(\text{pr}_{G}[Y_{\alpha},V_j],Y_1,\ldots,\widehat{Y_{\alpha}},\ldots,Y_u;V_1,\ldots,\widehat{V_j},\ldots,V_{v+1}\big)
	\end{align*}
	and
	\begin{align*}
		(d_{1,0}\omega)(Y_1,\ldots,Y_{u+1};V_1,\ldots,V_{v})&=\sum_{\alpha=1}^{u+1}(-1)^{\alpha+1}Y_{\alpha}\big(\omega(Y_1,\ldots,\widehat{Y_{\alpha}},\ldots,Y_{u+1};V_1,\ldots,V_v)\big)\\
		&\hspace{-2.5cm}+\sum_{1\leq\alpha<\beta\leq u+1}(-1)^{\alpha+\beta}\omega\big(\text{pr}_{G}[Y_{\alpha},Y_{\beta}],Y_1,\ldots,\widehat{Y_{\alpha}},\ldots,\widehat{Y_{\beta}},\ldots,Y_{u+1};V_1,\ldots,V_v\big)\\
		&\hspace{-2.5cm}+\sum_{\alpha=1}^{u+1}\sum_{j=1}^{v}(-1)^{\alpha+j+1}\omega\big(Y_1,\ldots,\widehat{Y_{\alpha}},\ldots,Y_{u+1};\text{pr}_{T\mathcal{F}}[Y_{\alpha},V_j],V_1,\ldots,\widehat{V_j},\ldots,V_v\big).
	\end{align*}

	\subsubsection{The proof of Prop.~\ref{prop:rephrase-canonical}}\label{subsubsec:can}
	
	To prove Prop.~\ref{prop:rephrase-canonical}, we rewrite the map $\Phi$ from Def.~\ref{phi} in terms of the bihomogeneous component
	$
	d_{1,0}:\Gamma(\wedge^{\bullet}T^{*}\mathcal{F})\rightarrow\Gamma(G^{*}\otimes\wedge^{\bullet}T^{*}\mathcal{F}).
	$
	More precisely, we have that
	\[
	\left\langle \Phi(\alpha)(V_{1},\ldots,V_{k}),\beta\right\rangle=(-1)^{k}(d_{1,0}\alpha)\big((\omega_C^{\flat})^{-1}(\beta),V_1,\ldots,V_k\big)
	\]
	for $\alpha\in\Gamma(\wedge^{k}T^{*}\mathcal{F}),\beta\in\Gamma(G^{*})$ and $V_1,\ldots,V_k\in\Gamma(T\mathcal{F})$. If we use the obvious identification
	\[
	\tau:\Gamma(G^{*}\otimes\wedge^{k}T^{*}\mathcal{F})\rightarrow\Gamma(\wedge^{k}T^{*}\mathcal{F}\otimes G^{*})
	\] 
	determined by
	\[
	\langle \tau(\eta)(V_1,\ldots,V_k),Y\rangle=\langle\eta(Y),V_1\wedge\cdots\wedge V_k\rangle,
	\]
	then we obtain the following description of the map $\Phi$.
	
	\begin{lemma}\label{lem:des-phi}
		The map $\Phi:\Gamma(\wedge^{k}T^{*}\mathcal{F})\rightarrow\Gamma(\wedge^{k}T^{*}\mathcal{F}\otimes G)$ is given by
	\begin{equation}\label{eq:expressphi}
		\Phi(\alpha)=(-1)^{k+1}\left(\text{Id}\otimes(\omega_C^{\flat})^{-1}\right)\big(\tau(d_{1,0}\alpha)\big),\hspace{1cm}\alpha\in\Gamma(\wedge^{k}T^{*}\mathcal{F}).
	\end{equation}

\end{lemma}

It follows that, in order to study $\Phi$, one only has to understand the map
	\[
	\tau\circ d_{1,0}:\Gamma(\wedge^{\bullet}T^{*}\mathcal{F})\rightarrow\Gamma(\wedge^{\bullet}T^{*}\mathcal{F}\otimes G^{*}).
	\]
We now show that this map intertwines the differentials $d_{\mathcal{F}}$ and $d_{\nabla^{*}}$.
	
	
	\begin{lemma}\label{lem:chainmap}
		We have a chain map
		\[
		\tau\circ d_{1,0}:\big(\Omega^{\bullet}(\mathcal{F}),d_{\mathcal{F}}\big)\rightarrow\big(\Omega^{\bullet}(\mathcal{F};G^{*}),d_{\nabla^{*}}\big).
		\]
	\end{lemma}
	\begin{proof}
		We start by expressing the complexes $\big(\Omega^{\bullet}(\mathcal{F}),d_{\mathcal{F}}\big)$ and $\big(\Omega^{\bullet}(\mathcal{F};G^{*}),d_{\nabla^{*}}\big)$ in the bi-degree language introduced above. We claim that the following hold.
		
		\noindent
		\begin{flalign*}
			\underline{\text{Claim:}}
			\ &i)\  \big(\Omega^{\bullet}(\mathcal{F}),d_{\mathcal{F}}\big)=\big(\Omega^{0,\bullet}(C),d_{0,1}\big),&\\
			&ii)\ \tau:\big(\Omega^{1,\bullet}(C),d_{0,1}\big)\overset{\sim}{\rightarrow}\big(\Omega^{\bullet}(\mathcal{F};G^{*}),d_{\nabla^{*}}\big)\ \text{intertwines differentials, up to sign.}
		\end{flalign*}
		
		\noindent
		Since $i)$ is clear, we only justify $ii)$. Pick $\beta\in\Omega^{1,k}(C)$, and let $V_1,\ldots,V_{k+1}\in\Gamma(T\mathcal{F})$ and $Y\in\Gamma(G)$. On one hand, we have
		\begin{align*}
			\big\langle\tau(d_{0,1}\beta)(V_1,\ldots,V_{k+1}),Y\big\rangle&=(d_{0,1}\beta)(Y,V_1,\ldots,V_{k+1})\\
			&=\sum_{i=1}^{k+1}(-1)^{i}V_i\big(\beta(Y,V_1,\ldots,\widehat{V_i},\ldots,V_{k+1})\big)\\
			&\hspace{0.5cm}+\sum_{1\leq i<j\leq k+1}(-1)^{i+j+1}\beta(Y,[V_i,V_j],V_1,\ldots,\widehat{V_i},\ldots,\widehat{V_j},\ldots,V_{k+1})\\
			&\hspace{0.5cm}+\sum_{j=1}^{k+1}(-1)^{j}\beta\big(\text{pr}_{G}[Y,V_j],V_1,\ldots,\widehat{V_j},\ldots,V_{k+1}\big).
		\end{align*}
		On the other hand,
		\begin{align*}
			\big\langle d_{\nabla^{*}}(\tau(\beta))(V_1,\ldots,V_{k+1}),Y \big\rangle&=\sum_{i=1}^{k+1}(-1)^{i+1}\big\langle\nabla^{*}_{V_i}\tau(\beta)(V_1,\ldots,\widehat{V_i},\ldots,V_{k+1}),Y \big\rangle\\
			&+\sum_{1\leq i<j\leq k+1}(-1)^{i+j}\big\langle\tau(\beta)\big([V_i,V_j],V_1,\ldots,\widehat{V_i},\ldots,\widehat{V_j},\ldots,V_{k+1}\big),Y \big\rangle.
		\end{align*}
		Here
		\begin{align*}
			\big\langle\nabla^{*}_{V_i}\tau(\beta)(V_1,\ldots,\widehat{V_i},\ldots,V_{k+1}),Y \big\rangle&=V_i\big\langle \tau(\beta)(V_1,\ldots,\widehat{V_i},\ldots,V_{k+1}),Y\big\rangle\\
			&\hspace{0.5cm}-\big\langle\tau(\beta)(V_1,\ldots,\widehat{V_i},\ldots,V_{k+1}),\text{pr}_{G}[V_i,Y]\big\rangle\\
			&=V_i\big(\beta(Y,V_1,\ldots,\widehat{V_i},\ldots,V_{k+1})\big)\\
			&\hspace{0.5cm}-\beta\big(\text{pr}_{G}[V_i,Y],V_1,\ldots,\widehat{V_i},\ldots,V_{k+1}\big)
		\end{align*}
		and
		\[
		\big\langle\tau(\beta)\big([V_i,V_j],V_1,\ldots,\widehat{V_i},\ldots,\widehat{V_j},\ldots,V_{k+1}\big),Y \big\rangle=\beta\big(Y,[V_i,V_j],V_1,\ldots,\widehat{V_i},\ldots,\widehat{V_j},\ldots,V_{k+1}\big).
		\]
		Upon comparison, we see that $\tau(d_{0,1}\beta)=-d_{\nabla^{*}}(\tau(\beta))$, hence the claim is proved.
		
		Next, we relate the bihomogeneous components of the de Rham differential $d$ with each other. By the decomposition \eqref{eq:components}, we have that $d^{2}$ maps $\Omega^{u,v}(C)$ into
		\[
		\Omega^{u,v+2}(C)\oplus\Omega^{u+1,v+1}(C)\oplus\Omega^{u+2,v}(C)\oplus\Omega^{u+3,v-1}(C)\oplus\Omega^{u+4,v-2}(C).
		\]
		Because each component of $d^{2}$ must be zero, we obtain
		\begin{equation}\label{eq:relations}
			\begin{cases}
				d_{0,1}d_{0,1}=0\\
				d_{0,1}d_{1,0}+d_{1,0}d_{0,1}=0\\
				d_{0,1}d_{2,-1}+d_{1,0}d_{1,0}+d_{2,-1}d_{0,1}=0\\
				d_{1,0}d_{2,-1}+d_{2,-1}d_{1,0}=0\\
				d_{2,-1}d_{2,-1}=0
			\end{cases}.
		\end{equation}
		
		At last, we prove the statement of the lemma. Using the second relation in \eqref{eq:relations} and $ii)$ of the claim above, we have for $\alpha\in\Gamma(\wedge^{\bullet}T^{*}\mathcal{F})$ that
		\[
		(\tau\circ d_{1,0})(d_{0,1}\alpha)=-(\tau\circ d_{0,1})( d_{1,0}\alpha)=d_{\nabla^{*}}((\tau\circ d_{1,0})(\alpha)).
		\]
		Since by $i)$ of the claim above, the restriction of $d_{0,1}$ to $\Gamma(\wedge^{\bullet}T^{*}\mathcal{F})$ coincides with the foliated de Rham differential $d_{\mathcal{F}}$, this finishes the proof.
	\end{proof}
	
	Item $i)$ of Prop.~\ref{prop:rephrase-canonical} is now immediate.
	
	\begin{cor}
	The map
	\[
	\Phi:\big(\Omega^{\bullet}(\mathcal{F}),d_{\mathcal{F}}\big)\rightarrow\big(\Omega^{\bullet}(\mathcal{F};G),d_{\nabla}\big)
	\]
	is a chain map, up to sign.
	\end{cor}
	\begin{proof}
	From the expression \eqref{eq:expressphi}, we get that for $\alpha\in\Omega^{k}(\mathcal{F})$, 
	\begin{align*}
		\Phi(d_{\mathcal{F}}\alpha)&=(-1)^{k}\left(\text{Id}\otimes(\omega_C^{\flat})^{-1}\right)\big(\tau(d_{1,0}(d_{\mathcal{F}}\alpha))\big)\\
		&=(-1)^{k}\left(\text{Id}\otimes(\omega_C^{\flat})^{-1}\right)\big(d_{\nabla^{*}}(\tau(d_{1,0}\alpha))\big)\\
		&=(-1)^{k}d_{\nabla}\left(\left(\text{Id}\otimes(\omega_C^{\flat})^{-1}\right)(\tau(d_{1,0}\alpha))\right)\\
		&=-d_{\nabla}(\Phi(\alpha)).
	\end{align*}
	Here we also used that $\text{Id}\otimes(\omega_C^{\flat})^{-1}$ intertwines $d_{\nabla^{*}}$ and $d_{\nabla}$, see \S\ref{subsub:fol}.
	\end{proof}
	
	Also item $ii)$ of Prop.~\ref{prop:rephrase-canonical} is a consequence of Lemma~\ref{lem:chainmap}.
	
	\begin{cor}\label{cor:transversediff}
	\begin{enumerate}
	\item The following is a chain map, up to sign:
		\[
		(-1)^{\bullet}(\tau\circ d_{1,0}):\big(\Omega^{\bullet}(\mathcal{F}),d_{\mathcal{F}}\big)\rightarrow\big(\Omega^{\bullet}(\mathcal{F};G^{*}),d_{\nabla^{*}}\big).
		\]
\item The map induced by $(-1)^{\bullet}(\tau\circ d_{1,0})$ in cohomology is canonical, since it coincides with the transverse differentiation map $
d_{\nu}:H^{\bullet}(\mathcal{F})\rightarrow H^{\bullet}(\mathcal{F};N^{*}\mathcal{F}).
$
\item The map $\Phi$ induces a canonical map in cohomology, namely
\[
\big[-\text{Id}\otimes(\omega_{C}^{\flat})^{-1}\big]\circ d_{\nu}:H^{\bullet}(\mathcal{F})\rightarrow H^{\bullet}(\mathcal{F};N\mathcal{F}).
\]
\end{enumerate}
\end{cor}
\begin{proof}
For $\alpha\in\Omega^{k}(\mathcal{F})$, we have
\[
(-1)^{k+1}(\tau\circ d_{1,0})(d_{\mathcal{F}}\alpha)=-d_{\nabla^{*}}\left((-1)^{k}(\tau\circ d_{1,0})(\alpha)\right),
\]
which proves $(1)$. To prove item $(2)$, we will identify $N^{*}\mathcal{F}\cong G^{*}$. For a closed foliated form $\alpha\in\Gamma(\wedge^{k}T^{*}\mathcal{F})$, note that one can compute $d_{\nu}[\alpha]$ as follows. First extend $\alpha$ by zero on the complement $G$, so that we can view $\alpha$ as an element of $\Omega^{k}(C)$. Then pick the component of 
\[
d\alpha\in\Gamma(\wedge^{k+1}T^{*}\mathcal{F})\oplus\Gamma(\wedge^{k}T^{*}\mathcal{F}\otimes G^{*})\oplus\Gamma(\wedge^{k-1}T^{*}\mathcal{F}\otimes\wedge^{2}G^{*})
\]
lying in $\Gamma(\wedge^{k}T^{*}\mathcal{F}\otimes G^{*})$. This immediately implies item $(2)$. Item $(3)$ is now a consequence of item $(2)$ and the expression \eqref{eq:expressphi}.
\end{proof}
	
We now showed that the canonical map $d_{\nu}$ in cohomology is induced by a cochain map, which can be defined at the expense of choosing a complement $G$ to $T\mathcal{F}$. The cochain map in question is essentially the component $d_{1,0}$. This implies that the space $\ker(d_{\nu})$ is a term on the page $E_2$ of the spectral sequence of the foliation $\mathcal{F}$. This gives more insight into the space of first order deformations, see Def.~\ref{def:firstorderdefs}. We give the details in the following remark.

\begin{remark}
Given an arbitrary foliation $\mathcal{F}$, recall that its spectral sequence $\{E_k,d_k\}$ arises from a descending filtration of the de Rham complex, given by the spaces
\[
\Omega^{r}_{k}=\big\{\omega\in\Omega^{r}(C):\iota_{X_1\wedge\cdots\wedge X_{r-k+1}}\omega=0,\ \ \forall X_1,\ldots,X_{r-k+1}\in\Gamma(T\mathcal{F})\big\}.
\] 
Picking a complement $G$ to $T\mathcal{F}$, we see that $\Omega^{r}_{k}$ is the ideal in $\Omega^{r}(C)$ generated by $\Gamma(\wedge^{k}G^{*})$. The page $E_0$ has terms given by
\[
E_0^{u,v}=\Omega^{u+v}_{u}/\Omega^{u+v}_{u+1}\cong\Gamma(\wedge^{u}G^{*}\otimes\wedge^{v}T^{*}\mathcal{F}),
\]
and it is equipped with the differential $d_{0,1}$ defined in \eqref{eq:components}. More concisely, $E_0=\big(\Omega(C),d_{0,1}\big)$. The page $E_1$ is the cohomology of $E_0$, endowed with the differential induced by $d_{1,0}$, i.e.
\[
\big(H(\Omega(C),d_{0,1}),[d_{1,0}]\big).
\]
Using implicitly the identifications from the claim in Lemma \ref{lem:chainmap}, we see that the first two columns of the page $E_1$ look as follows:
	\begin{equation*}
		\begin{tikzcd}[row sep=small]
			&\vdots & \vdots\\
			&H^{2}(\mathcal{F})\arrow[r,"{[d_{1,0}]}"]&H^{2}(\mathcal{F};G^{*})\arrow[r,"{[d_{1,0}]}"]&\cdots\\
			&H^{1}(\mathcal{F})\arrow[r,"{[d_{1,0}]}"]&H^{1}(\mathcal{F};G^{*})\arrow[r,"{[d_{1,0}]}"]&\cdots\\
			&H^{0}(\mathcal{F})\arrow[r,"{[d_{1,0}]}"]&H^{0}(F;G^{*})\arrow[r,"{[d_{1,0}]}"]&\cdots  
		\end{tikzcd}
	\end{equation*}
The page $E_2$ consists of the cohomology groups of the complex $E_1$, i.e.
\[
H\left(H(\Omega(C),d_{0,1}),[d_{1,0}]\right).
\]
Also using Cor.~\ref{cor:transversediff}, we see in particular that
\begin{equation}\label{eq:page2}
	E^{0,1}_{2}\cong\ker\big(d_{\nu}:H^{1}(\mathcal{F})\rightarrow H^{1}(\mathcal{F};N^{*}\mathcal{F})\big). 
\end{equation}

	This point of view explains in a more conceptual way why the first order deformations in Def.~\ref{def:firstorderdefs} reduce in the integral case to those considered by Ruan \cite{ruan}. We checked this fact by direct computation in Ex.~\ref{ex:integral}. A simpler way is remarking that, in case $\mathcal{F}$ is given by the fibers of a fiber bundle, then the spectral sequence of $\mathcal{F}$ reduces to the Leray spectral sequence of $C\rightarrow C/\mathcal{F}$. Indeed, if $C$ is compact and $\mathcal{F}$ is given by a fiber bundle, then the Leray-Serre theorem \cite[Cor.~I.5.3]{hoster} states that
	\[
	E^{0,1}_{2}\cong H^{0}\big(C/\mathcal{F};\mathcal{H}^{1}\big),
	\]
	where $\mathcal{H}^{1}$ is the flat vector bundle defined in \eqref{eq:vb}. Therefore, $E^{0,1}_{2}$ can be identified with the space of flat sections of $(\mathcal{H}^{1},\nabla)$, showing that Def.~\ref{def:firstorderdefs} is consistent with Ruan's work.
	
\end{remark}

\subsubsection{The proof of Lemma~\ref{lem:extension}}
At last, we give the proof of Lemma~\ref{lem:extension}. It relies on the results proved in \S\ref{subsubsec:can}, namely on the proof of Lemma~\ref{lem:chainmap} and on Cor.~\ref{cor:transversediff}.

	\begin{proof}[Proof of Lemma~\ref{lem:extension}]
		Fix a complement $G$ to $T\mathcal{F}$. By the claim in the proof of Lemma~\ref{lem:chainmap} and Cor.~\ref{cor:transversediff} $(2)$, we have to show that the following two statements are equivalent:
		\begin{enumerate}
			\item $d_{0,1}\beta=0$ and $d_{1,0}\beta$ is exact in $\big(\Omega^{1,\bullet}(C),d_{0,1}\big)$.
			\item There exists an extension $\widetilde{\beta}\in\Omega^{1}(C)$ of $\beta$ such that $\iota_{V}d\widetilde{\beta}=0$ for all $V\in\Gamma(T\mathcal{F})$.
		\end{enumerate}
		For one implication, assume that $d_{0,1}\beta=0$ and that $d_{1,0}\beta=d_{0,1}\gamma$ for some $\gamma\in\Gamma(G^{*})$. Then $\widetilde{\beta}:=\beta-\gamma\in\Gamma(T^{*}\mathcal{F}\oplus G^{*})$ is an extension of $\beta$ satisfying
		\begin{align*}
			d\widetilde{\beta}&= d_{0,1}\beta+d_{1,0}\beta+d_{2,-1}\beta -d_{0,1}\gamma-d_{1,0}\gamma\\
			&=d_{2,-1}\beta-d_{1,0}\gamma.
		\end{align*}
		The latter belongs to $\Gamma(\wedge^{2}G^{*})$, and therefore $\iota_{V}d\widetilde{\beta}=0$ for all $V\in\Gamma(T\mathcal{F})$.
		
		For the converse, we can write the given extension as $\widetilde{\beta}=\beta+\gamma$ for some $\gamma\in\Gamma(G^{*})$. Decomposing $d\widetilde{\beta}$ in the direct sum
		\[
		\Omega^{2}(C)=\Gamma(\wedge^{2}T^{*}\mathcal{F})\oplus\Gamma(G^{*}\otimes T^{*}\mathcal{F})\oplus\Gamma(\wedge^{2}G^{*}),
		\]
		we have
		\[
		d\widetilde{\beta}=d_{0,1}\beta+\big(d_{1,0}\beta+d_{0,1}\gamma\big)+\big(d_{2,-1}\beta+d_{1,0}\gamma\big).
		\]
		Since by assumption, the components in $\Gamma(\wedge^{2}T^{*}\mathcal{F})$ and $\Gamma(G^{*}\otimes T^{*}\mathcal{F})$ vanish, we get
		\[
		d_{0,1}\beta=0\ \ \text{and}\ \ d_{1,0}\beta=-d_{0,1}\gamma.\qedhere
		\]
	\end{proof}

	\subsection{A computation in the Gotay local model}
	\vspace{0.1cm}
	\noindent
	
	Given a coisotropic submanifold $C\subset(M,\omega)$ with characteristic foliation $\mathcal{F}$, consider the Gotay local model $(U,\Omega_G)$ associated with a choice of splitting $TC=T\mathcal{F}\oplus G$. Throughout this paper, we frequently identified vertical fiberwise constant vector fields on $U$ with sections of $T^{*}\mathcal{F}$. In Lemma~\ref{dF} below, we give the explicit formulae underlying this correspondence. This result is essentially \cite[Lemma~4.4]{equivalences}, up to an additional minus sign. To justify this difference, we include a detailed proof in coordinates.

	As before, $p:U\rightarrow C$ denotes the projection and  $j:T^{*}\mathcal{F}\hookrightarrow T^{*}C$ the inclusion induced by $G$. Also, $\Pi_G=-\Omega_G^{-1}$ is the Poisson structure corresponding with $\Omega_G$, $i:C\hookrightarrow U$ is the inclusion map and $r:T^{*}C\rightarrow T^{*}\mathcal{F}$ is the restriction to the leaves of $\mathcal{F}$.

	\begin{lemma}\label{dF}
		There is a natural correspondence
		\[
		\Gamma(T^{*}\mathcal{F})\rightarrow\mathfrak{X}_{vert.const.}(U):\beta \mapsto \Pi_{G}^{\sharp}(p^{*}(j(\beta)))
		\]
		with inverse
		\[
		\mathfrak{X}_{vert.const.}(U)\rightarrow\Gamma(T^{*}\mathcal{F}):V\mapsto -r\left(i^{*}\Omega_{G}^{\flat}(V)\right).
		\]
	\end{lemma}
	\begin{proof}
		Choose coordinates $(q_1,\ldots,q_k,q_{k+1},\ldots,q_n)$ on $C$ such that $T\mathcal{F}=\text{Span}\{\partial_{q_1},\ldots,\partial_{q_k}\}$. Let $(y_1,\ldots,y_n)$ denote the conjugate coordinates on $T^{*}C$. We have to show that the correspondence described above reads as follows in these coordinates:
		\begin{equation}\label{checkcorr}
			\sum_{i=1}^{k}g_i(q)dq_i\leftrightarrow\sum_{i=1}^{k}g_i(q)\partial_{y_i}.
		\end{equation}
		We start by expressing the symplectic form $\Omega_G$ in coordinates.
		Since $G$ is transverse to the first summand in the decomposition
		\[
		TC=\text{Span}\{\partial_{q_1},\ldots,\partial_{q_k}\}\oplus\text{Span}\{\partial_{q_{k+1}},\ldots,\partial_{q_n}\},
		\]
		there exists a fiberwise linear map $\Psi$ such that
		\begin{align*}
			G&=\text{Graph}\left(\Psi:\text{Span}\{\partial_{q_{k+1}},\ldots,\partial_{q_n}\}\rightarrow\text{Span}\{\partial_{q_1},\ldots,\partial_{q_k}\}\right)\\
			&=\text{Span}\left\{\partial_{q_{k+1}}+\Psi(\partial_{q_{k+1}}),\ldots,\partial_{q_n}+\Psi(\partial_{q_n})\right\}.
		\end{align*}
		Let us write for $l=k+1,\ldots,n$:
		\[
		\Psi(\partial_{q_l})=\sum_{i=1}^{k}f_{l}^{i}(q)\partial_{q_i}.
		\]
		We then obtain for $i=1,\ldots,k$ that
		\begin{align*}
			\Big\langle j(dq_i),\sum_{l=1}^{n}h_l(q)\partial_{q_l}\Big\rangle&=\Big\langle dq_i,\sum_{l=1}^{k}h_l(q)\partial_{q_l}-\sum_{l=k+1}^{n}h_l(q)\Psi(\partial_{q_l})\Big\rangle\\
			&=\Big\langle dq_i,\sum_{l=1}^{k}h_l(q)\partial_{q_l}-\sum_{l=k+1}^{n}\sum_{\alpha=1}^{k}f_{l}^{\alpha}(q)h_l(q)\partial_{q_{\alpha}}\Big\rangle\\
			&=h_i(q)-\sum_{l=k+1}^{n}f_{l}^{i}(q)h_l(q).
		\end{align*}
		This shows that
		\[
		j(dq_i)=dq_i-\sum_{l=k+1}^{n}f_l^{i}(q)dq_l,
		\]
		and therefore
		\[
		j\left(\sum_{i=1}^{k}y_idq_i\right)=\sum_{i=1}^{k}y_idq_i-\sum_{l=k+1}^{n}\left(\sum_{i=1}^{k}y_if_{l}^{i}(q)\right)dq_l.
		\]
		Hence, expressing the map $j:T^{*}\mathcal{F}\hookrightarrow T^{*}C$ in coordinates gives
		\[
		j(q_1,\ldots,q_n,y_1,\ldots,y_k)=\left(q_1,\ldots,q_n,y_1,\ldots,y_k,-\sum_{i=1}^{k}y_if_{k+1}^{i}(q),\ldots,-\sum_{i=1}^{k}y_if_{n}^{i}(q)\right).
		\]
		It follows that
		\begin{align*}
			j^{*}\omega_{can}&=\sum_{i=1}^{k}dq_i\wedge dy_i - \sum_{l=k+1}^{n}dq_l\wedge d\left(\sum_{i=1}^{k}y_if_l^{i}(q)\right)\\
			&=\sum_{i=1}^{k}dq_i\wedge dy_i-\sum_{l=k+1}^{n}\sum_{i=1}^{k}f_l^{i}(q)dq_l\wedge dy_i -\sum_{l=k+1}^{n}\sum_{i=1}^{k}y_i dq_l\wedge df_{l}^{i}(q),
		\end{align*}
		hence the symplectic form $\Omega_G$ looks like
		\[
		\Omega_G=\sum_{k+1\leq i<j\leq n}\omega_{ij}(q)dq_i\wedge dq_j+\sum_{i=1}^{k}dq_i\wedge dy_i-\sum_{l=k+1}^{n}\sum_{i=1}^{k}f_l^{i}(q)dq_l\wedge dy_i -\sum_{l=k+1}^{n}\sum_{i=1}^{k}y_i dq_l\wedge df_{l}^{i}(q).
		\]
		It is now clear that the correspondence described in the statement of the lemma acts as required in \eqref{checkcorr}, since we have
		\begin{align*}
			-\Omega_{G}^{\flat}\left(\sum_{i=1}^{k}g_i(q)\partial_{y_i}\right)&=\sum_{i=1}^{k}g_i(q)dq_i-\sum_{l=k+1}^{n}\left(\sum_{i=1}^{k}g_i(q)f_l^{i}(q)\right)dq_l\\
			&=p^{*}\left(j\left(\sum_{i=1}^{k}g_i(q)dq_i\right)\right).\qedhere
		\end{align*}
	\end{proof}

	\begin{remark}\label{rem:corres}
		The correspondence from Lemma~\ref{dF} extends to foliated differential forms and vertical fiberwise constant multivector fields. Explicitly, it is given by
		\[
		\Gamma(\wedge^{k}T^{*}\mathcal{F})\rightarrow\mathfrak{X}^{k}_{vert.const.}(U):\beta\mapsto(\wedge^{k}\Pi_G^{\sharp})(p^{*}(j(\beta)))
		\]
		with inverse
		\[
		\mathfrak{X}^{k}_{vert.const.}(U)\rightarrow\Gamma(\wedge^{k}T^{*}\mathcal{F}):V\mapsto(-1)^{k}r\left(i^{*}(\wedge^{k}\Omega_G^{\flat})(V)\right).
		\]
	\end{remark}
	
	\section*{Declarations}
	
	\subsection*{Funding}
	During the preparation of this article, the author was funded by the Max Planck Institute for Mathematics in Bonn and the UCL Institute for Mathematics and Statistical Science (IMSS).
	
	\subsection*{Competing interests}
	The author has no competing interests to declare that are relevant to the content of this article.

	\subsection*{Data availability statement}
	Data sharing not applicable to this article as no datasets
	were generated or analysed during the current study.


\begin{thebibliography}{20}
		
		\bibitem{arraut} J.L. Arraut and N.M. dos Santos, \textit{Linear foliations of $\mathbb{T}^n$}, Bol. Soc. Bras. Mat. \textbf{21}(2), p. 189-204, 1991.
		
		\bibitem{foliations} M. Asaoka, A. El Kacimi-Alaoui, S. Hurder and K. Richardson, \textit{Foliations: Dynamics, Geometry and Topology}, Advanced Courses in Mathematics CRM Barcelona, Birkhäuser, Springer Basel, 2014.
		
		\bibitem{bolle} P. Bolle, \textit{A contact condition for $p$-codimensional submanifolds of a symplectic manifold $(2\leq p\leq n)$}, Math. Z. \textbf{227}(2), p. 211-230, 1998.
		
		\bibitem{candelconlon} A. Candel and L. Conlon, \textit{Foliations I}, Graduate Studies in Mathematics, Volume 23, American Mathematical Society, Providence, Rhode Island, 2000.
		
		\bibitem{cannas} A. Cannas da Silva, \textit{Lectures on Symplectic Geometry}, Lecture Notes in Mathematics, Volume 1764, Springer-Verlag Berlin, 2001.
		
		\bibitem{modular} R. Caseiro and R.L. Fernandes, \textit{The modular class of a Poisson map}, Ann. Inst. Fourier \textbf{63}(4), p. 1285-1329, 2013.
		
		\bibitem{cattaneo-felder} A. Cattaneo and G. Felder, \textit{Relative formality theorem and quantisation of coisotropic submanifolds}, Adv. Math. \textbf{208}(2), p. 521-548, 2007.
		
		\bibitem{crainic-fernandes} M. Crainic and R.L. Fernandes, \textit{Integrability of Poisson brackets}, J. Differential Geom. \textbf{66}(1), p. 71-137, 2004.
		
		\bibitem{symplmatrix} M. de Grosson, \textit{Introduction to Symplectic Mechanics: Lectures I-II-III}, 2006. \\ Available at \url{https://www.ime.usp.br/~piccione/Downloads/LecturesIME.pdf}.
		
		\bibitem{libermann} J.P. Dufour and N.T. Zung, \textit{Poisson Structures and Their Normal Forms}, Progress in Mathematics, Volume 242, Birkh{\"a}user Verlag Basel, 2005.
		
		\bibitem{anosov} A. El Kacimi-Alaoui and M. Nicolau, \textit{A class of $C^{\infty}$-stable foliations}, Ergod. Theory Dyn. Syst. \textbf{13}(4), p. 697-704, 2008. 
		
		\bibitem{b-coiso} S. Geudens and M. Zambon, \textit{Coisotropic submanifolds in b-symplectic geometry}, Canad. J. Math. \textbf{73}(3), p. 737–768, 2021.
		
		\bibitem{gotay} M.J. Gotay, \textit{On coisotropic imbeddings of presymplectic manifolds}, Proc. Amer. Math. Soc. \textbf{84}(1), p. 111-114, 1982.	
		
		\bibitem{heitsch} J.L. Heitsch, \textit{A cohomology for foliated manifolds}, Comment. Math. Helv. \textbf{50}(1), p. 197-218, 1975.
		
		\bibitem{hoster} M. Hoster, \textit{Derived Secondary Classes for Flags of Foliations}, Ph.D. thesis, Ludwig-Maximilians-Universit\"{a}t M\"{u}nchen, 2001.
		
		\bibitem{lada} T. Lada and J. Stasheff, \textit{Introduction to SH Lie algebras for physicists}, Int. J. Theor. Phys. \textbf{32}(7), p. 1087-1103, 1993.
		
		
		\bibitem{mcduff} D. McDuff and D. Salamon, \textit{Introduction to Symplectic Topology}, Oxford Graduate Texts in Mathematics, Volume 27, Oxford University Press, 2017.
		
		\bibitem{oh-park} Y.-G. Oh and Y.-S. Park, \textit{Deformations of coisotropic submanifolds and strong homotopy Lie algebroids}, Invent. Math. \textbf{161}(2), p. 287-360, 2005.
		
		\bibitem{osorno} B. Osorno-Torres, \textit{Codimension-one Symplectic Foliations: Constructions and Examples}, Ph.D. thesis, Utrecht University, 2015.
		
		\bibitem{ruan} W.-D. Ruan, \textit{Deformation of integral coisotropic submanifolds in symplectic manifolds}, J. Symplectic Geom. \textbf{3}(2), p. 161-169, 2005.
		
		\bibitem{fiberwise} F. Sch{\"a}tz and M. Zambon, \textit{Deformations of coisotropic submanifolds for fibrewise entire Poisson structures}, Lett. Math. Phys. \textbf{103}(7), p. 777-791, 2013.
		
		\bibitem{equivalences} F. Sch{\"a}tz and M. Zambon, \textit{Equivalences of coisotropic submanifolds}, J. Symplectic Geom. \textbf{15}(1), p. 107-149, 2017.
		
		\bibitem{koszul} F. Sch{\"a}tz and M. Zambon, \textit{Deformations of pre-symplectic structures and the Koszul $L_\infty$-algebra}, Int. Math. Res. Not. \textbf{2020}(14), p. 4191-4237, 2020.
		
		\bibitem{rel} F. Sch{\"a}tz and M. Zambon, \textit{From coisotropic to presymplectic structures: relating the deformation theories}, In preparation.
		
		\bibitem{tondeur} P. Tondeur, \textit{Geometry of Foliations}, Monographs in Mathematics, Volume 90, Springer Basel AG, 1997.
		
		\bibitem{lagr} A. Weinstein, \textit{Symplectic manifolds and their Lagrangian submanifolds}, Adv. Math. \textbf{6}(3), p. 329-346, 1971.
		
		
		\bibitem{marco} M. Zambon, \textit{An example of coisotropic submanifolds $C^{1}$-close to a given coisotropic submanifold}, Differ. Geom. Appl. \textbf{26}(6), p. 635-637, 2008.
		
	\end{thebibliography}
\end{document}